\documentclass[12pt]{article}

\usepackage{amsmath, amsthm, amssymb, amsfonts, bm, enumitem, hyperref}

\usepackage{stmaryrd}
\usepackage{bbold}
\usepackage{times}
\usepackage{mathrsfs}
\usepackage{graphicx}
\usepackage{tikz}
\usetikzlibrary{matrix,calc,arrows, trees}

\newcommand{\QQ}{\mathbb{Q}}

\newcommand{\kk}{\mathbb{k}}

\DeclareMathOperator{\Hom}{Hom}
\DeclareMathOperator{\End}{End}

\DeclareMathOperator{\Aut}{Aut}
\DeclareMathOperator{\hAut}{hAut}

\DeclareMathOperator{\Der}{Der}

\DeclareMathOperator{\id}{id}

\DeclareMathOperator{\Cyl}{Cyl}
\DeclareMathOperator{\Gr}{Gr}

\DeclareMathOperator{\CH}{CH}

\DeclareMathOperator{\Ger}{Ger}

\DeclareMathOperator{\wt}{wt}
\DeclareMathOperator{\res}{res}
\DeclareMathOperator{\incl}{incl}

\DeclareMathOperator{\Isom}{Isom}

\DeclareMathOperator{\colim}{colim}

\DeclareMathOperator{\Cobar}{Cobar}
\DeclareMathOperator{\susp}{\bf{s}}

\newcommand{\onto}{\twoheadrightarrow}

\newcommand{\tto}{\longrightarrow}

\newcommand{\fg}{\mathfrak{g}}

\newcommand{\bt}{\textbf{t}}

\newcommand{\sC}{\mathcal{C}}

\newcommand{\sF}{\mathcal{F}}

\newcommand{\sH}{\mathcal{H}}

\newcommand{\sM}{\mathcal{M}}
\newcommand{\sO}{\mathcal{O}}

\newcommand{\sQ}{\mathcal{Q}}

\renewcommand{\k}{\mathbb{k}}

\newcommand{\cat}[1]{\mathsf{#1}}

\newcommand{\cCh}{\cat{Ch_{\k}}}
\newcommand{\Tree}{\cat{Tree}}
\newcommand{\PF}{\cat{PF}}

\newcommand{\del}{\partial}
\newcommand{\ot}{\otimes}

\newcommand{\OP}{\mathbb{OP}}

\newtheorem{theorem}{Theorem}[section]
\newtheorem{lemma}[theorem]{Lemma}
\newtheorem{proposition}[theorem]{Proposition}
\newtheorem{corollary}[theorem]{Corollary}

\newenvironment{remark}[1][Remark.]{\begin{trivlist}
\item[\hskip \labelsep {\bfseries #1}]}{\end{trivlist}}

\newtheorem*{theorem*}{Theorem}
\newtheorem*{lemma*}{Lemma}
\newtheorem*{proposition*}{Proposition}
\newtheorem*{corollary*}{Corollary}
\newtheorem*{conjecture*}{Conjecture}

\usepackage{fullpage}

\begin{document}

\begin{center}
\section*{Action of derived automorphisms on infinity-morphisms}
\large{Brian Paljug}
\end{center}

\begin{abstract}

In this paper we investigate how to simultaneously change homotopy algebras of a certain type and a corresponding infinity morphism between them, and show that this can be done in a homotopically unique way. More precisely, for a reduced cooperad $\sC$, given ${\Cobar}(\sC)$-algebras $V$ and $W$ and an $\infty$-morphism $U: V \leadsto W$, and for the automorphism $\exp(D)$ for any degree $0$ closed $D \in \Der'({\Cobar}(\sC))$, we produce new ${\Cobar}(\sC)$-algebras and a new $\infty$-morphism $\widetilde{U} : V^D \leadsto W^D$ extending $U$ in a universal way. Operads play the central role in answering this question, in particular a $2$-colored operad $\Cyl(\sC)$ that governs pairs of $\infty$-algebras and $\infty$-morphisms between them.


\end{abstract}

\section{Introduction}

Homotopy algebras and morphisms appear in many areas throughout mathematics; in homological algebra in the form of algebraic transfer theorems, in geometry in the study of iterated loop spaces, in deformation quantization in Kontsevich's formality theorem, and so on. Much work has been done to find the correct framework in which to study homotopy algebras, and the theory of operads is one such attempt. The complicated coherence relations that define homotopy algebras are encoded in the language of operads, which are easily manipulated with homological or combinatorial techniques; see \cite{AlgOps} for an excellent overview of these techniques. This is the approach we take in this paper.

While there is a notion of morphisms between homotopy algebras of a specific type, in practice and theory one is interested in the looser notions of $\infty$-morphisms, which themselves satisfy some complicated system of coherence relations. Since homotopy algebras can be defined as algebras over a specific operad, it seems natural to ask if $\infty$-morphisms can be defined in the language of operads. The answer is provided in \cite{OpCobarCyl} via a $2$-colored ``cylinder construction'' operad, which we restate and study further in this paper; similar ideas were also considered in \cite{HomDiagAlgs}, \cite{ForDefMor} and \cite{ResDiagAlg}. Specifically, given a cooperad $\sC$ we construct and study a $2$-colored operad $\Cyl(\sC)$ that governs pairs of homotopy algebras and $\infty$-morphisms between them. 

Our main goal is to answer the following question; given a pair of homotopy algebras and an $\infty$-morphism between them, can we change the homotopy algebras and the $\infty$-morphism simultaneously to get new homotopy algebras and a new $\infty$-morphism (all of the same type)? More specifically, given a derivation of the operad ${\Cobar}(\sC)$ governing the homotopy algebras $V$ and $W$, we can exponentiate that derivation to an automorphism of ${\Cobar}(\sC)$ and use that automorphism to define new ${\Cobar}(\sC)$-algebra structures on $V$ and $W$ via pullback; can we do the same to an $\infty$-morphism between $V$ and $W$, to create a new $\infty$-morphism that respects the new ${\Cobar}(\sC)$-algebra structures? We show that this is possible, and moreover that the answer is unique up to homotopy, using the previously mentioned techniques of operadic homological algebra. In particular, we have Theorem \ref{der-q-isom}, which is summarized below.

\begin{theorem*}
The maps
\begin{center}
\begin{tabular}{c c c c}
$\res_\alpha, \res_\beta :$ & $\Der(\Cyl(\sC))$ & $\tto$ & $\Der({\Cobar}(\sC))$
\end{tabular}
\end{center}
given by restricting to a single color $\alpha$ or $\beta$ are homotopic quasi-isomorphisms of dg Lie algebras at all filtration levels.
\end{theorem*}
\noindent This leads us to Theorem \ref{res-isom-h}:
\begin{theorem*}
The group homomorphisms
\begin{center}
\begin{tabular}{c c c c}
$\res_\alpha, \res_\beta :$ & $\Aut'(\Cyl(\sC))$ & $\tto$ & $\Aut'({\Cobar}(\sC))$
\end{tabular}
\end{center}
induce identical isomorphisms on homotopy classes:
\begin{center}
\begin{tabular}{c c c c}
$\res :$ & $\hAut'(\Cyl(\sC))$ & $\tto$ & $\hAut'({\Cobar}(\sC))$.
\end{tabular}
\end{center}
\end{theorem*}
\noindent This then allows us to answer the motivating question, shown in Theorem \ref{main-thm}:
\begin{theorem*}
Let $V$ and $W$ be ${\Cobar}(\sC)$-algebras for a cooperad $\sC$, and let $U: V \leadsto W$ be an $\infty$-morphism between them. Given a degree $0$ closed derivation $D \in \Der'({\Cobar} (\sC))$, there exists a degree $0$ cocycle $\widetilde{D} \in \Der'(\Cyl(\sC))$ such that $D$, $\widetilde{D}_\alpha$, and $\widetilde{D}_\beta$ are cohomologous in $\Der'(\Cobar(\sC))$. Therefore we can construct
\begin{align*}
U^{\widetilde{D}}: V^{\widetilde{D}_\alpha} \leadsto W^{\widetilde{D}_\beta}
\end{align*}
such that $V^{\widetilde{D}_\alpha}$ is homotopy equivalent to $V^D$ and $W^{\widetilde{D}_\beta}$ is homotopy equivalent to $W^D$, and so that the linear term of $U$ is unchanged: $U^{\widetilde{D}}_{(0)} = U_{(0)}$.
\end{theorem*}

In the subsequent paper \cite{TamConGRT} with V. Dolgushev, we show how our results may be applied to justify a claim made in Section 10.2 of Thomas Willwacher's paper \cite{KonGrGRT}, concerning the action of the Grothendieck-Teichm\"{u}ller group $\text{GRT}_1$ on formality morphisms. The Grothendieck-Teichm\"{u}ller group and Lie algebra are connected to Drinfeld associators, the absolute Galois group of $\QQ$, Kontsevich's graph complex, the homotopy theory of the operad governing Gerstenhaber algebras, and so on. This application is explored more fully in the paper \cite{TamConGRT}.

The author is very grateful to Vasily Dolgushev, who was instrumental to the creation and completion of this paper, and patiently endured (and continues to endure) many hours of questions and explanations. The results of this paper were first presented at the Higher Structure 2013: Operads and Deformation Theory conference held at the Isaac Newton Institue in Cambridge, UK. The author is grateful for the chance to discuss this work at that conference, and in particular would like to thank Bruno Valette, Christopher Rogers, and Thomas Willwacher for their comments and conversation. The author is partially supported by NSF grants DMS-0856196 and DMS-1161867.

\section{Preliminaries}

First, we set notation and recall some basic definitions and facts. For a general introduction to the theory of operads, see \cite{NotesAlgOps} or \cite{AlgOps}. Throughout, we work in the category of cochain complexes of graded vector spaces over a field $\k$ of characteristic zero, $\cCh$. We will freely use the abbreviation ``dg'' to stand for the phrase ``differential graded.'' Given an operad $\sO$,
\begin{align*}
\mu: \sO(n) \ot \sO(k_1) \ot ... \ot \sO(k_n) \to \sO(k_1+...+k_n)
\end{align*}
will denote operadic multiplication, while
\begin{align*}
\circ_i : \sO(n) \ot \sO(k) \to \sO(n+k-1)
\end{align*}
denotes the $i$-th elementary insertion, as usual. Dually, given a cooperad $\sC$,
\begin{align*}
\Delta : \sC(k_1+...+k_n) \to \sC(n) \ot \sC(k_1) \ot ... \ot \sC(k_n)
\end{align*}
denotes the cooperadic comultiplication, while
\begin{align*}
\Delta_i : \sC(n+k-1) \to \sC(n) \ot \sC(k)
\end{align*}
denotes the $i$-th elementary coinsertion. Occasionally, the more general (co)multiplications 
\begin{align*}
\mu_\bt : \underline{\sO}_n(\bt) \to \sO(n) \hspace{1in} \Delta_\bt : \sC(n) \to \underline{\sC}_n(\bt)
\end{align*}
from \cite{NotesAlgOps} will be needed, where $\underline{\sO}_n$ is the functor associated to the collection $\sO$ from the category $\Tree(n)$ of $n$-labeled rooted planar trees to $\cCh$, and $\bt \in \Tree(n)$ (likewise for $ \underline{\sC}_n$ and $\sC$). Later, we will need two subcategories of $\Tree(n)$. The first is $\Tree_2(n)$, the full subcategory of $\Tree(n)$ consisting of trees with exactly $2$ internal vertices. The second is $\PF_k(n)$ (``pitchforks''), the full subcategory of $\Tree(n)$ consisting of trees with exactly $k+1$ internal vertices, one of which has height $1$, and the other $k$ have height exactly $2$. Some examples of such trees can be found in Figures \ref{extree} and \ref{expf}.

\begin{figure}[h]
\noindent\makebox[\textwidth][c]{
\begin{minipage}{.5\textwidth}
\centering
\tikzstyle{node} = [circle, draw, solid, text centered]
\tikzstyle{leaf} = [circle, minimum width=3pt,fill, inner sep=0pt]
\tikzstyle{level 1}=[level distance=1cm, sibling distance=1cm]
\tikzstyle{level 2}=[level distance=1.5cm, sibling distance=1cm]
\tikzstyle{level 3}=[level distance=1.5cm, sibling distance=.8cm]
\begin{tikzpicture}[grow'=up]
\node[leaf] {}
	child {
		node[node] {}        
		child {
			node[node] {}  
			child {
				node[leaf, label=above:{\footnotesize{$1$}}] {}     
				edge from parent 
			}
			child {
				node[leaf, label=above:{\footnotesize{$2$}}] {}        
				edge from parent 
			}
			child {
				node[leaf, label=above:{\footnotesize{$3$}}] {}        
				edge from parent 
			}
			edge from parent   
		}
		child {
			node[leaf, label=above:{\footnotesize{$4$}}] {}        
			edge from parent 
		}
		child {
			node[leaf, label=above:{\footnotesize{$5$}}] {}        
			edge from parent 
		}
		child {
			node[leaf, label=above:{\footnotesize{$6$}}] {}        
			edge from parent 
		}
		edge from parent         
	};
\end{tikzpicture}
\caption{An element of $\Tree_2(6)$.}\label{extree}
\end{minipage}
\begin{minipage}{.5\textwidth}
\centering
\tikzstyle{node} = [circle, draw, solid, text centered]
\tikzstyle{leaf} = [circle, minimum width=3pt,fill, inner sep=0pt]
\tikzstyle{level 1}=[level distance=1cm, sibling distance=1cm]
\tikzstyle{level 2}=[level distance=1.5cm, sibling distance=1.8cm]
\tikzstyle{level 3}=[level distance=1.5cm, sibling distance=.8cm]
\begin{tikzpicture}[grow'=up]
\node[leaf] {}
	child {
		node[node] {}        
		child {
			node[node] {}  
			child {
				node[leaf, label=above:{\footnotesize{$1$}}] {}     
				edge from parent 
			}
			child {
				node[leaf, label=above:{\footnotesize{$2$}}] {}        
				edge from parent 
			}
			child {
				node[leaf, label=above:{\footnotesize{$3$}}] {}        
				edge from parent 
			}
			edge from parent   
		}
		child {
			node[node] {}  
			child {
				node[leaf, label=above:{\footnotesize{$4$}}] {}     
				edge from parent 
			}
			child {
				node[leaf, label=above:{\footnotesize{$5$}}] {}        
				edge from parent 
			}
			edge from parent   
		}
		child {
			node[node] {}
			child {
				node[leaf, label=above:{\footnotesize{$6$}}] {}     
				edge from parent 
			}
			child {
				node[leaf, label=above:{\footnotesize{$7$}}] {}        
				edge from parent 
			}
			child {
				node[leaf, label=above:{\footnotesize{$8$}}] {}        
				edge from parent 
			}
			edge from parent 
		}
		edge from parent         
	};
\end{tikzpicture}
\caption{An element of $\PF_3(8)$.}\label{expf}
\end{minipage}
}
\end{figure}

If an operad is augmented, $\sO_\circ$ will denote the kernel of the augmentation map, while $\sC_\circ$ will denote the cokernel of the coaugmentation map of a coaugmented cooperad. Throughout, $\del_\sO$ will denote the differential of an operad $\sO$, and likewise $\del_\sC$ for a cooperad $\sC$. A coaugmented cooperad $\sC$ is called \emph{reduced} if
\begin{align*}
\sC(0)=\{0\} \hspace{1in} \sC(1)=\k
\end{align*}
and hence $\sC_\circ(0)=\sC_\circ(1)=\{0\}$. We will assume that all cooperads are reduced for the remainder of the paper. 

The endomorphism operad of a dg vector space $V$, $\End_V$, is of particular importance, since it allows us to define algebras over an operad. We have $\End_V(n)=\Hom(V^{\ot n}, V)$, with operadic multiplication defined by function composition, and the symmetric action defined by rearranging tensor factors. Then we say that $V$ is an algebra over the operad $\sO$ (or that $V$ is an $\sO$-algebra) if we have a map of operads
\begin{align*}
\sO \to \End_V
\end{align*}
or, equivalently, we have multiplication maps
\begin{align*}
\mu_n : \sO(n) \ot V^{\ot n} \to V
\end{align*}
for all $n \geq 0$, satisfying appropriate associativity, equivariance, and unit axioms \cite{OpsPROPs}. This second formulation leads to the definition of $\sC$-coalgebras: we say that $V$ is a coalgebra over the cooperad $\sC$ if we have comultiplication maps
\begin{align*}
\Delta_n : V \to \sC(n) \ot V^{\ot n}
\end{align*}
satisfying appropriate dual axioms. This then leads to the definitions of the free $\sO$-algebra
\begin{align*}
\sO(V) = \bigoplus_{n \geq 0} \left(\sO(n) \ot V^{\ot n}\right)_{S_n}
\end{align*}
and the cofree $\sC$-coalgebra
\begin{align*}
\sC(V) = \bigoplus_{n \geq 0} \left(\sC(n) \ot V^{\ot n}\right)^{S_n}.
\end{align*}

Given a collection $\sQ$, we can form the free operad $\OP(\sQ)$. It is generally most convenient to think of elements of $\OP(\sQ)$ as rooted trees with internal vertices decorated by elements of $\sQ$, subject to an appropriate symmetry relation; with this in mind, $(\bt; x_1, ..., x_k)$ is the element of $\OP(\sQ)(n)$ where the $n$-labeled tree $\bt$ has $k$ internal vertices decorated by the elements $x_1,...,x_k$ of $\sM$, according to the total ordering on internal vertices. We will frequently identify an element $x \in \sQ(n)$ with the standard $n$-corolla decorated by $x$ in $\OP(\sQ)(n)$. It is occasionally useful to remember that at each level $\OP(\sQ)(n)=\colim \underline{\sQ}_n$ (strictly, speaking, this only defines the free pseudo-operad, with the free operad then obtained by formally adjoining a unit). Since the $S_n$ action on $\OP(\sQ)(n)$ permutes the labels, in this paper we will omit these labels when drawing elements of $\OP(\sQ)$.

Then, given a coaugmented cooperad $\sC$, we can define the cobar construction $\Cobar ( \sC)$ to be $\OP(\susp\sC_\circ)$ as an operad of graded vector spaces (where $\susp$ is the suspension operator of degree $1$), with the differential defined on generators $\susp x \in \susp\sC_\circ$ by
\begin{align*}
\del_{\Cobar} (\susp x) = -\susp\del_\sC(x)-\sum_{z \in \Isom( \Tree_2(n) )} (-1)^{|x_1|} (\bt_z; \susp x_1, \susp x_2)
\end{align*}
where the sum is taken over all isomorphism classes of $\Tree_2(n)$, $\bt_z$ is a representative of the isomorphism class $z\in \Isom( \Tree_2(n))$, and $\Delta_{\bt_z}(x)=\sum x_1 \ot x_2$. Note that we use Sweedler type notation in the above equation, and will continue to do so throughout the paper.

Much of the above notation extends immediately to the colored setting, so we will only focus on certain ideas and notation; our primary reference is \cite{StableI}. We will focus on the $2$-colored setting for now, and the $3$-colored versions should be clear (and will only be minimally needed in this paper). We will refer to our $2$ colors as $\alpha$ and $\beta$. Given a $2$-colored collection $\sQ$, $\sQ(a,b;\alpha)$ denotes the level of $\sQ$ with $a$ inputs of color $\alpha$, $b$ inputs of color $\beta$, and output of color $\alpha$. Similarly, $\sQ(a,b;\beta)$ indicates that the output is of color $\beta$. As in the single-color case, we have a free colored operad construction, governed by colored trees. The category of colored $n$-labeled trees is defined similarly to $\Tree(n)$, except that the edges carry colors, and morphisms must respect the coloring. In figures, edges with color $\alpha$ will be represented by solid lines, while edges of color $\beta$ will be represented by dashed lines.

As before, we will need two subcategories of colored $n$-labeled trees, but slightly more specialized than those given above. The first is $\Tree'_2(n)$, the full subcategory of $\Tree(n)$ consisting of trees with exactly $2$ internal vertices, such that the root edge carries color $\beta$, and all other edges carry color $\alpha$. The second is $\PF'_k(n)$, the full subcategory of colored $\Tree(n)$ consisting of trees with exactly $k+1$ internal vertices, one of which has height $1$, and the other $k$ have height exactly $2$, such that all leaf edges carry color $\alpha$ and all other edges carry color $\beta$. Some examples of such trees can be found in Figures \ref{excoltree} and \ref{excolpf}.

\begin{figure}[h]
\noindent\makebox[\textwidth][c]{
\begin{minipage}{.5\textwidth}
\centering
\tikzstyle{node} = [circle, draw, solid, text centered]
\tikzstyle{leaf} = [circle, minimum width=3pt,fill, inner sep=0pt]
\tikzstyle{level 1}=[level distance=1cm, sibling distance=1cm]
\tikzstyle{level 2}=[level distance=1.5cm, sibling distance=1cm]
\tikzstyle{level 3}=[level distance=1.5cm, sibling distance=.8cm]
\begin{tikzpicture}[grow'=up]
\node[leaf] {}
	child {
		node[node] {}        
		child {
			node[node] {}  
			child {
				node[leaf, label=above:{\footnotesize{$1$}}] {}     
				edge from parent [solid]
			}
			child {
				node[leaf, label=above:{\footnotesize{$2$}}] {}        
				edge from parent [solid]
			}
			child {
				node[leaf, label=above:{\footnotesize{$3$}}] {}        
				edge from parent [solid]
			}
			edge from parent [solid] 
		}
		child {
			node[leaf, label=above:{\footnotesize{$4$}}] {}        
			edge from parent [solid]
		}
		child {
			node[leaf, label=above:{\footnotesize{$5$}}] {}        
			edge from parent [solid]
		}
		child {
			node[leaf, label=above:{\footnotesize{$6$}}] {}        
			edge from parent [solid]
		}
		edge from parent [dashed]
	};
\end{tikzpicture}
\caption{An element of $\Tree'_2(6)$.}\label{excoltree}
\end{minipage}
\begin{minipage}{.5\textwidth}
\centering
\tikzstyle{node} = [circle, draw, solid, text centered]
\tikzstyle{leaf} = [circle, minimum width=3pt,fill, inner sep=0pt]
\tikzstyle{level 1}=[level distance=1cm, sibling distance=1cm]
\tikzstyle{level 2}=[level distance=1.5cm, sibling distance=1.8cm]
\tikzstyle{level 3}=[level distance=1.5cm, sibling distance=.8cm]
\begin{tikzpicture}[grow'=up]
\node[leaf] {}
	child {
		node[node] {}        
		child {
			node[node] {}  
			child {
				node[leaf, label=above:{\footnotesize{$1$}}] {}     
				edge from parent [solid]
			}
			child {
				node[leaf, label=above:{\footnotesize{$2$}}] {}        
				edge from parent [solid]
			}
			child {
				node[leaf, label=above:{\footnotesize{$3$}}] {}        
				edge from parent [solid]
			}
			edge from parent   
		}
		child {
			node[node] {}  
			child {
				node[leaf, label=above:{\footnotesize{$4$}}] {}     
				edge from parent [solid]
			}
			child {
				node[leaf, label=above:{\footnotesize{$5$}}] {}        
				edge from parent [solid]
			}
			edge from parent   
		}
		child {
			node[node] {}
			child {
				node[leaf, label=above:{\footnotesize{$6$}}] {}     
				edge from parent [solid]
			}
			child {
				node[leaf, label=above:{\footnotesize{$7$}}] {}        
				edge from parent [solid]
			}
			child {
				node[leaf, label=above:{\footnotesize{$8$}}] {}        
				edge from parent [solid]
			}
			edge from parent 
		}
		edge from parent [dashed]
	};
\end{tikzpicture}
\caption{An element of $\PF'_3(8)$.}\label{excolpf}
\end{minipage}
}
\end{figure}

We will briefly recall the definitions of homotopy algebras and $\infty$-morphisms; see \cite{NotesAlgOps}, \cite{HomAlgOps}, and \cite{AlgOps} for more thorough introductions to the subject. Given a coaugmented, reduced cooperad $\sC$, we will use the following ``pedestrian'' definition of homotopy algebras: a homotopy algebra of type $\sC$ to be an algebra $V$ over ${\Cobar}(\sC)$. That is, we have a map of operads
\begin{align*}
F: {\Cobar}(\sC) \tto \End_V.
\end{align*}
The complicated systems of coherence relations needed in explicit definitions of homotopy algebras are built into the compatibility of the above map with the differentials. This is equivalent \cite{NotesAlgOps} to a coderivation $Q_V$ on $\sC(V)$, the cofree coalgebra generated by $V$ over the cooperad $\sC$, that satisfies the Maurer-Cartan equation (equivalently, $d_V + Q_V$ is a differential on $\sC(V)$).

While there is the natural notion of morphisms of $\sO$-algebras for any operad $\sO$, in this setting we have the richer notion of $\infty$-morphisms. More specifically, an $\infty$-morphism between two homotopy algebras $V$, $W$ of type $\sC$ is a map of dg coalgebras
\begin{align*}
U: (\sC(V), d_V+Q_V) \tto (\sC(W), d_W+Q_W).
\end{align*}
We denote an $\infty$-morphism by $U: V \leadsto W$. An $\infty$-quasi-isomorphism is an $\infty$-morphism such that the linear term $U_{(0)}: V \to W$ is a quasi-isomorphism of complexes. In the event that $V$ and $W$ are $\infty$-quasi-isomorphic, we will say that $V$ and $W$ are homotopy equivalent.

\section{The $2$-colored operad $\Cyl(\sC)$}

As mentioned in the introduction, our ultimate goal is to study homotopy algebras and $\infty$-morphisms using the techniques of operads. In particular, we will define and study a $2$-colored operad $\Cyl(\sC)$ that governs pairs of $\Cobar(\sC)$-algebras and an $\infty$-morphism $V \leadsto W$ (see Proposition \ref{prop:cyl-alg}).

First, given a coaugmented cooperad $\sC$, define the $2$-colored collection $\widetilde{\sC}$ as follows:
\[ \widetilde{\sC}(n,0;\alpha) = \widetilde{\sC}(0,n;\beta) = \susp\sC_\circ(n) \]
\[ \widetilde{\sC}(n,0;\beta) = \sC(n) \]
\[ \widetilde{\sC} = 0 \hspace{.2in} \text{otherwise.} \]

\noindent Then form the free $2$-colored operad $\OP(\widetilde{\sC})$. We think of elements of $\OP(\widetilde{\sC})$ as $2$-colored trees with vertices decorated by elements of $\susp \sC_\circ$ and $\sC$, such that vertices have incoming edges only of a single color, and vertices have input color $\alpha$ and output color $\beta$ exactly when the vertex is decorated by an element of $\sC$.

Note that unlike ${\Cobar}(\sC)$, elements of $\OP(\widetilde{\sC})$ may have mixed-color vertices decorated by unsuspended elements of $\sC$, and since $\sC=\k \oplus \sC_\circ$, in particular may have vertices with a single input and output. We denote these ``trivial vertices'' by $1^{\alpha\beta} \in \OP(\widetilde{\sC})(1,0;\beta)$. This leads us to an alternate notion of degree: given $X \in \OP(\widetilde{\sC})$ (or $X \in {\Cobar}(\sC)$), we will say that the \emph{weight} of $X$, or $\wt(X)$, is the number of internal vertices of $X$ not of the form $1^{\alpha\beta}$ (in ${\Cobar}(\sC)$, weight is just the number of internal vertices). We will say that a map $F$ has weight $m$ if it raises weight by exactly $m$. Note finally that we have ${\Cobar}(\sC) \subseteq \OP(\widetilde{\sC})$ by declaring that an element $X \in {\Cobar}(\sC)$ has edges only of color $\alpha$ or only of color $\beta$; we will denote these assignments by $X^{\alpha}$ and $X^{\beta}$, respectively. Similarly, given $x\in\sC(n)$, we will write $x^{\alpha\beta} \in \OP(\widetilde{\sC})(n,0;\beta)$ to indicate the corolla with $n$ incoming edges of color $\alpha$ and outgoing edge of color $\beta$; this is consistent with our earlier notation for $1^{\alpha\beta} \in \OP(\widetilde{\sC})(1,0;\beta)$. We will often mark vectors with superscripts in this way to clearly indicate their input and output colors.


On $\OP(\widetilde{\sC})$, define a derivation $\del$ on generators as follows:
\begin{center}
\begin{tabular}{c c c c}
$\del(\susp x^\alpha)$ & $=$ & $\del_{\Cobar} (\susp x)^\alpha$ & $\susp x \in \widetilde{\sC}(n,0;\alpha)= \susp\sC_\circ(n)$ \\
$\del(\susp x^\beta)$ & $=$ & $\del_{\Cobar} (\susp x)^\beta$ & $\susp x \in \widetilde{\sC}(0,n;\beta)= \susp\sC_\circ(n)$ \\
$\displaystyle\del(1^{\alpha\beta})$ & $=$ & $ 0$ & $1 \in \widetilde{\sC}(n,0;\beta)= \sC(n)$ \\
$\displaystyle\del(x^{\alpha\beta})$ & $=$ & $ \del_\sC (x)^{\alpha\beta} + \del'(x^{\alpha\beta})+\del''(x^{\alpha\beta})$ & $1 \neq x \in \widetilde{\sC}(n,0;\beta)= \sC(n)$
\end{tabular}
\end{center}
where $\del'$ is defined by
\begin{align*}
\del'(x^{\alpha\beta}) = \displaystyle \sum_{z \in \Isom( \Tree'_2(n,0;\beta) )} (-1)^{|x_1|} (\bt_z; x_1, \susp x_2)
\end{align*}
with $\bt_z$ a representative of the isomorphism class $z\in \Isom( \Tree'_2(n))$ and $\Delta_{\bt_z}(x)=\sum x_1 \ot x_2$, and where $\del''$ is defined by
\begin{align*}
\del''(x^{\alpha\beta}) = \displaystyle -\sum_k \sum_{z \in \Isom( \PF'_k(n,0;\beta) )} (\bt_z; \susp x_0, x_1, ..., x_k)
\end{align*}
with $\bt_z$ a representative of the isomorphism class $z \in \Isom( \PF'_k(n,0;\beta))$ and $\Delta_{\bt_z}(x)=\sum x_0 \ot x_1 \ot ... \ot x_k$, and where both comultiplications are taken by forgetting the coloring on $\bt_z$. A visual interpretation of $\del(x^{\alpha\beta})$  is found in Figure \ref{scary-diff}.

{
\begin{figure}[h]
\noindent\makebox[\textwidth][c]{
\begin{minipage}{.1\textwidth}
\centering
$\del(x^{\alpha\beta})=$
\end{minipage}
\begin{minipage}{.2\textwidth}
\centering
\tikzstyle{node} = [circle, draw, solid, text centered]
\tikzstyle{leaf} = [circle, minimum width=3pt,fill, inner sep=0pt]
\tikzstyle{level 1}=[level distance=2cm, sibling distance=1cm]
\tikzstyle{level 2}=[level distance=2.2cm, sibling distance=1cm]
\begin{tikzpicture}[grow'=up]
\node[leaf] {}
	child {
		node[node] {\footnotesize{$\del_{\sC}(x)$}}        
		child {
			node[leaf] {}  
			edge from parent [solid] 
		}
		child {
			node[leaf] {}        
			edge from parent [solid]
		}
		child {
			edge from parent [white]
			node[black] {...}
		}
		child {
			node[leaf] {}        
			edge from parent [solid]
		}
		edge from parent [dashed]
	};
\end{tikzpicture}
\end{minipage}
\begin{minipage}{.05\textwidth}
\centering
$+$
\end{minipage}
\begin{minipage}{.2\textwidth}
\centering
$\displaystyle\sum_{\Isom( \Tree'_2(n,0;\beta) )} (-1)^{|x_1|}$
\end{minipage}
\begin{minipage}{.3\textwidth}
\centering
\tikzstyle{node} = [circle, draw, solid, text centered]
\tikzstyle{leaf} = [circle, minimum width=3pt,fill, inner sep=0pt]
\tikzstyle{level 1}=[level distance=1.2cm, sibling distance=1cm]
\tikzstyle{level 2}=[level distance=1.5cm, sibling distance=1cm]
\tikzstyle{level 3}=[level distance=1.5cm, sibling distance=.8cm]
\begin{tikzpicture}[grow'=up]
\node[leaf] {}
	child {
		node[node] {\footnotesize{$x_1$}}     
		child {
			node[node] {\footnotesize{$\susp x_2$}}
			child {
				node[leaf] {}     
				edge from parent 
			}
			child {
				node[leaf] {}        
				edge from parent 
			}
			child {
				edge from parent [white]
				node[black] {...}
			}
			child {
				node[leaf] {}        
				edge from parent 
			}
			edge from parent [solid]
		}
		child {
			node[leaf] {}        
			edge from parent [solid]
		}
		child {
			edge from parent [white]
			node[black] {...}
		}
		child {
			node[leaf] {}        
			edge from parent [solid]
		}
		edge from parent [dashed]
	};
\end{tikzpicture}
\end{minipage}
}
\vspace{3mm}
 
\noindent\makebox[\textwidth][c]{
\begin{minipage}{.2\textwidth}
\centering
$-\displaystyle\sum_k \sum_{\Isom( \PF'_k(n,0;\beta) )}$
\end{minipage}
\begin{minipage}{.35\textwidth}
\centering
\tikzstyle{node} = [circle, draw, solid, text centered]
\tikzstyle{leaf} = [circle, minimum width=3pt,fill, inner sep=0pt]
\tikzstyle{level 1}=[level distance=1.2cm, sibling distance=1cm]
\tikzstyle{level 2}=[level distance=1.5cm, sibling distance=1.8cm]
\tikzstyle{level 3}=[level distance=1.5cm, sibling distance=.8cm]
\begin{tikzpicture}[grow'=up]
\node[leaf] {}
	child {
		node[node] {\footnotesize{$\susp x_0$}}      
		child {
			node[node] {\footnotesize{$x_1$}}
			child {
				node[leaf] {}     
				edge from parent [solid]
			}
			child {
				node[leaf] {}        
				edge from parent [solid]
			}
			child {
				edge from parent [white]
				node[black] {...}
			}
			child {
				node[leaf] {}        
				edge from parent [solid]
			}
			edge from parent   
		}
		child {
			node[] {...}
			edge from parent [white]
		}
		child {
			node[node] {\footnotesize{$x_k$}}
			child {
				node[leaf] {}     
				edge from parent [solid]
			}
			child {
				node[leaf] {}        
				edge from parent [solid]
			}
			child {
				edge from parent [white]
				node[black] {...}
			}
			child {
				node[leaf] {}        
				edge from parent [solid]
			}
			edge from parent 
		}
		edge from parent [dashed]       
	};
\end{tikzpicture}
\end{minipage}
}
\caption{The differential on $\Cyl(\sC)$.}\label{scary-diff}
\end{figure}
}

$\del$ visibly has degree $1$, and so with the following proposition, we see that $\del$ gives $\OP(\widetilde{\sC})$ the structure of a dg operad. Following \cite{OpCobarCyl}, we will call this operad $\Cyl(\sC)$.

\begin{proposition}
$\del^2=0$.
\end{proposition}
\begin{proof}
The proof is a technical computation in the same spirit as showing $\del_{\Cobar}^2=0$. It suffices to show that $\del^2=0$ on corollas. Since $\del = \del_{\Cobar}$ on single-color corollas, it remains to justify that $\del^2=0$ on mixed-color corollas; we will give the general ideas behind this computation. Since $\del = \del_\sC + \del' + \del''$, we have that $\del^2=0$ from the following observations:
\begin{enumerate}
\item $\del_\sC^2=0$ because $\del_\sC$ is a differential on $\sC$;
\item $\del_\sC \circ \del' +\del' \circ \del_\sC=\del_\sC \circ \del'' +\del'' \circ \del_\sC=0$ because $\del_\sC$ is as a coderivation;
\item $\del' \circ \del' = 0$ because of coassociativity;
\item $\del' \circ \del'' + \del'' \circ \del' + \del'' \circ \del''= 0$ because of coassociativity and elementary combinatorial identities.
\end{enumerate}
\end{proof}


As mentioned earlier, the significance of $\Cyl(\sC)$ is that it governs pairs of homotopy algebras and $\infty$-morphisms between them, which we will prove in Section 4; for now, we proceed to study $\Cyl(\sC)$ in more depth. Given that $\Cyl(\sC)$ is essentially a $2$-colored modification of ${\Cobar}(\sC)$, one would expect their cohomology to be related somehow. This is indeed the case, at least if we restrict our attention to the weight $0$ components of their respective differentials. The weight $0$ component of $\del_{\Cobar}$ is just $\del_\sC$; explicitly,
\begin{align*}
\del_\sC(\susp x) = -\susp\del_\sC(x)
\end{align*}
for $\susp x \in \susp\sC_\circ$. On $\Cyl(\sC)$, the weight $0$ part of $\del$, to be denoted $\del_0$, is given explicitly by
\begin{center}
\begin{tabular}{c c c c}
$\del_0(\susp x^{\alpha})$ & $=$ & $\del_\sC (\susp x)^\alpha$ & $\susp x \in \widetilde{\sC}(n,0;\alpha)= \susp\sC_\circ(n)$ \\
$\del_0(\susp x^{\beta})$ & $=$ & $\del_\sC (\susp x)^\beta$ & $\susp x \in \widetilde{\sC}(0,n;\beta)= \susp\sC_\circ(n)$ \\
$\displaystyle\del(1^{\alpha\beta})$ & $=$ & $ 0$ & $1 \in \widetilde{\sC}(n,0;\beta)= \sC(n)$ \\
$\del_0(x^{\alpha\beta})$ & $=$ & $\del_\sC (x)^{\alpha\beta} + \del'_0(x^{\alpha\beta}) + \del''_0(x^{\alpha\beta})$ & $1 \neq x \in \widetilde{\sC}(n,0;\beta)= \sC(n)$
\end{tabular}
\end{center}
where
\begin{align*}
\del'_0(x^{\alpha\beta}) = 1^{\alpha\beta} \circ_1 \susp x^\alpha
\end{align*}
and where
\begin{align*}
\del''_0(x^{\alpha\beta}) = - \mu(\susp x^\beta; 1^{\alpha\beta}, ..., 1^{\alpha\beta}).
\end{align*}
It may seem as though we have overrused the notation $\del_\sC$ by now, but all such uses are really just the original $\del_\sC$ acting as a derivation on a free operad, respecting suspensions and/or coloring. A visual representation of the action of $\del_0$ on mixed-color generators is found in Figure \ref{scary-diff-0}.

\begin{figure}[h]
\noindent\makebox[\textwidth][c]{
\begin{minipage}{.1\textwidth}
\centering
$\del_0(x^{\alpha\beta})=$
\end{minipage}
\begin{minipage}{.2\textwidth}
\centering
\tikzstyle{node} = [circle, draw, solid, text centered]
\tikzstyle{leaf} = [circle, minimum width=3pt,fill, inner sep=0pt]
\tikzstyle{level 1}=[level distance=1.7cm, sibling distance=1cm]
\tikzstyle{level 2}=[level distance=2cm, sibling distance=1cm]
\begin{tikzpicture}[grow'=up]
\node[leaf] {}
	child {
		node[node] {\footnotesize{$\del_{\sC}(x)$}}        
		child {
			node[leaf] {}  
			edge from parent [solid] 
		}
		child {
			node[leaf] {}        
			edge from parent [solid]
		}
		child {
			edge from parent [white]
			node[black] {...}
		}
		child {
			node[leaf] {}        
			edge from parent [solid]
		}
		edge from parent [dashed]
	};
\end{tikzpicture}
\end{minipage}
\begin{minipage}{.05\textwidth}
\centering
$+$
\end{minipage}
\begin{minipage}{.18\textwidth}
\centering
\tikzstyle{node} = [circle, draw, solid, text centered]
\tikzstyle{leaf} = [circle, minimum width=3pt,fill, inner sep=0pt]
\tikzstyle{level 1}=[level distance=1cm, sibling distance=1cm]
\tikzstyle{level 2}=[level distance=1.2cm, sibling distance=1cm]
\tikzstyle{level 3}=[level distance=1.5cm, sibling distance=.8cm]
\begin{tikzpicture}[grow'=up]
\node[leaf] {}
	child {
		node[node] {\footnotesize{$1$}}     
		child {
			node[node] {\footnotesize{$\susp x$}}
			child {
				node[leaf] {}     
				edge from parent 
			}
			child {
				node[leaf] {}        
				edge from parent 
			}
			child {
				edge from parent [white]
				node[black] {...}
			}
			child {
				node[leaf] {}        
				edge from parent 
			}
			edge from parent [solid]
		}
		edge from parent [dashed]
	};
\end{tikzpicture}
\end{minipage}
\begin{minipage}{.05\textwidth}
\centering
$-$
\end{minipage}
\begin{minipage}{.32\textwidth}
\centering
\tikzstyle{node} = [circle, draw, solid, text centered]
\tikzstyle{leaf} = [circle, minimum width=3pt,fill, inner sep=0pt]
\tikzstyle{level 1}=[level distance=1cm, sibling distance=1cm]
\tikzstyle{level 2}=[level distance=1.2cm, sibling distance=1.5cm]
\tikzstyle{level 3}=[level distance=1.5cm, sibling distance=.8cm]
\begin{tikzpicture}[grow'=up]
\node[leaf] {}
	child {
		node[node] {\footnotesize{$\susp x$}}      
		child {
			node[node] {\footnotesize{$1$}}
			child {
				node[leaf] {}     
				edge from parent [solid]
			}
			edge from parent   
		}
		child {
			node[node] {\footnotesize{$1$}}
			child {
				node[leaf] {}     
				edge from parent [solid]
			}
			edge from parent   
		}
		child {
			node[] {...}
			edge from parent [white]
		}
		child {
			node[node] {\footnotesize{$1$}}
			child {
				node[leaf] {}     
				edge from parent [solid]
			}
			edge from parent 
		}
		edge from parent [dashed]       
	};
\end{tikzpicture}
\end{minipage}
}
\caption{The weight $0$ component of the differential on $\Cyl(\sC)$.}\label{scary-diff-0}
\end{figure}
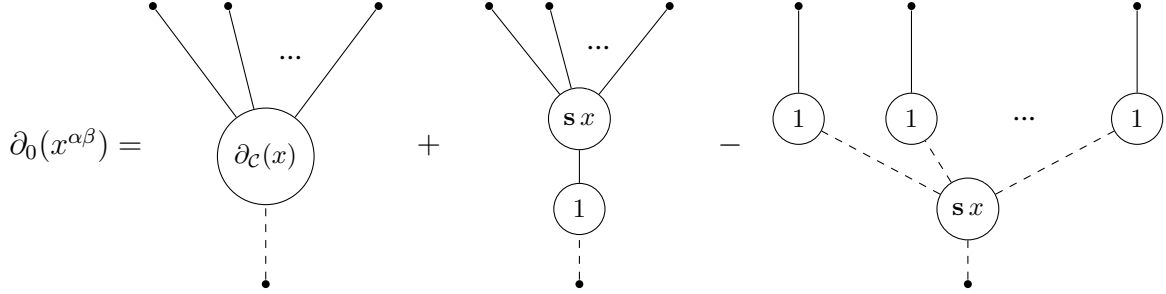

From weight considerations (or directly checking), both $\del_\sC^2=0$ and $\del_0^2=0$, so we may consider ${\Cobar} (\sC)$ and $\Cyl(\sC)$ with respect to these simpler differentials. Then we have:

\begin{theorem}\label{incl-q-isom}
The inclusion maps
\begin{center}
\begin{tabular}{c c c c c}
$\iota_\alpha$, $\iota_\beta$ & $:$ & $({\Cobar} (\sC)(n), \del_{\sC})$ & $\tto$ & $ (\Cyl(\sC)(n,0;\beta), \del_0 )$ \\
$\iota_\alpha$ & $:$ & $X$ & $\mapsto$ & $1^{\alpha\beta} \circ_1 X^\alpha$ \\
$\iota_\beta$ & $:$ & $X$ & $\mapsto$ & $\mu(X^\beta; 1^{\alpha\beta}, ..., 1^{\alpha\beta})$
\end{tabular}
\end{center}
are quasi-isomorphisms for all $n \geq 0$, and furthermore, are homotopic.
\end{theorem}

\begin{proof}
Given that we will show that $\iota_\alpha$ is homotopic to $\iota_\beta$, it suffices to show that $\iota_\beta$ is a quasi-isomorphism; we will begin with this. Introduce the following filtrations on ${\Cobar}(\sC)(n)$ and $\Cyl(\sC)(n,0;\beta)$:
\begin{center}
$\sF^m {\Cobar}(\sC)(n) = \left\{
\begin{array}{c}
X \in {\Cobar}(\sC)(n) \hspace{2mm} |\\
(\text{the number of edges in $X$}) - |X| \leq m
\end{array}
\right\}$
\end{center}
\begin{center}
$\sF^m \Cyl(\sC)(n,0;\beta) = \left\{
\begin{array}{c}
X \in \Cyl(\sC)(n,0;\beta) \hspace{2mm} |\\
(\text{the number of edges of color $\alpha$ in $X$}) - |X| \leq m
\end{array}
\right\}.$
\end{center}
These filtrations are ascending, cocomplete, and compatible with $\iota_\beta$ (since they are essentially the same filtration). They also respect $\del_\sC$ and $\del_0$; in particular, note that $\del_\sC$ and $\del_0''$ raise internal degree without changing the number of (straight) edges, so they lower the filtration index, while $\del_0'$ raises internal degree and the number of straight edges, so it preserves the filtration index. Consequently, when we consider the associated graded complexes, we have
\begin{center}
\begin{tabular}{c c c}
$\Gr_\sF {\Cobar}(\sC)(n)$ & $=$ & $({\Cobar}(\sC)(n), 0)$ \\
$\Gr_\sF \Cyl(\sC)(n,0;\beta)$ & $=$ & $(\Cyl(\sC)(n,0;\beta), \del_0')$
\end{tabular}
\end{center}
By Appendix A of \cite{NotesAlgOps}, it suffices to show that $\iota_\beta: ({\Cobar}(\sC)(n), 0) \to (\Cyl(\sC)(n,0;\beta), \del_0')$ is a quasi-isomorphism. For the remainder of this first section of the proof, when we refer to those complexes, they will carry those differentials.

For this, we need an auxiliary construction. Define the $3$-colored collection $\sQ$, with colors $\alpha$, $\beta$, $\gamma$, by
\begin{center}
\begin{tabular}{c c c c}
$\sQ(a,0,0;\alpha)$ & $=$ & $\susp\sC_\circ(a)$ & with $\del_\sQ=0$ \\
$\sQ(0,b,c;\beta)$ & $=$ & $\susp\sC_\circ(b+c)$ & with $\del_\sQ=0$ \\
$\sQ(a,0,0;\beta)$ & $=$ & $\sC_\circ(a) \oplus \susp\sC_\circ(a)$ & with $\del_\sQ: x \to \susp x$ \\
$\sQ$ & $=$ & $0$ & otherwise.
\end{tabular}
\end{center}
Note that $H^\bullet(\sQ(0,b,c;\beta)) = H^\bullet(\sQ(b+c,0,0;\alpha)) = \susp\sC_\circ (b+c)$, while $H^\bullet (\sQ(a,0,0;\beta))=0$.

When we form $\OP(\sQ)$, we have that
\begin{align*}
\displaystyle \Cyl(\sC)(n,0;\beta) \cong \bigoplus_{m=0}^n \OP(\sQ)(m,0,n-m;\beta)
\end{align*}
via the (backwards) identification $\OP(\sQ)(m,0,n-m;\beta) \to \Cyl(\sC)(n,0;\beta)$ determined by the following rules. First, send edges of color $\gamma$ to the element $1^{\alpha\beta} \in \Cyl(\sC)(1,0;\beta)$. Then perform the following identifications:
\begin{center}
\begin{tabular}{c c c}
$\susp x^{\alpha} \in \sQ(a,0,0;\alpha)$ & $\mapsto$ & $\susp x^{\alpha} \in \widetilde{\sC}(a,0;\alpha)$ \\
$\susp x^{\beta} \in \sQ(0,b,0;\beta)$ & $\mapsto$ & $\susp x^{\beta} \in \widetilde{\sC}(0,b;\beta)$ \\
$x^{\alpha\beta} \in \sC_\circ \subseteq \sQ(a,0,0;\beta)$ & $\mapsto$ & $x^{\alpha\beta} \in \widetilde{\sC}(a,0;\beta)$\\
$\susp x^{\alpha\beta} \in \susp\sC_\circ \subseteq \sQ(a,0,0;\beta)$ & $\mapsto$ & $1^{\alpha\beta} \circ_1 \susp x^\alpha  \in \Cyl(\sC)(a,0;\beta)$ \\
\end{tabular}
\end{center}
An example of this identification is shown in Figure ~\ref{3-color-version}.

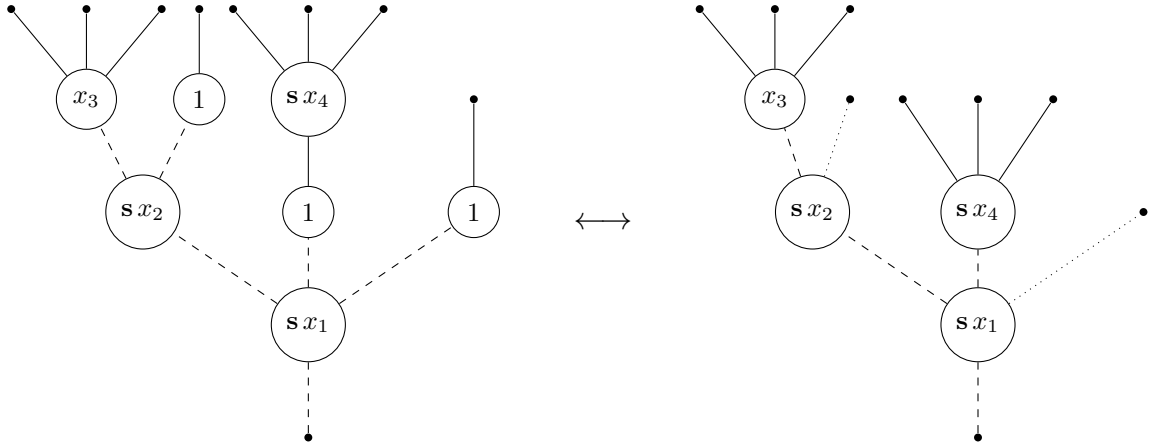
\begin{figure}[h]
\noindent\makebox[\textwidth][c]{
\begin{minipage}{.5\textwidth}
\centering
\tikzstyle{node} = [circle, draw, solid, text centered]
\tikzstyle{leaf} = [circle, minimum width=3pt,fill, inner sep=0pt]
\tikzstyle{level 1}=[level distance=1.5cm, sibling distance=1cm]
\tikzstyle{level 2}=[level distance=1.5cm, sibling distance=2.2cm]
\tikzstyle{level 3}=[level distance=1.5cm, sibling distance=1.5cm]
\tikzstyle{level 4}=[level distance=1.2cm, sibling distance=1cm]
\begin{tikzpicture}[grow'=up]
\node[leaf] {}
	child {
		node[node] {\footnotesize{$\susp x_1$}}        
		child {
			node[node] {\footnotesize{$\susp x_2$}}  
			child {
				node[node] {\footnotesize{$x_3$}}     
				child {
					node[leaf] {}     
					edge from parent [solid]
				}
				child {
					node[leaf] {}     
					edge from parent [solid]
				}
				child {
					node[leaf] {}     
					edge from parent [solid]
				}
				edge from parent 
			}
			child {
				node[node] {\footnotesize{$1$}}
				child {
					node[leaf] {}        
					edge from parent [solid]
				}
				edge from parent
			}
			edge from parent   
		}
		child {
			node[node] {\footnotesize{$1$}}
			child {
				node[node] {\footnotesize{$\susp x_4$}}     
				child {
					node[leaf] {}     
					edge from parent 
				}
				child {
					node[leaf] {}     
					edge from parent 
				}
				child {
					node[leaf] {}     
					edge from parent 
				}
				edge from parent [solid]
			}        
			edge from parent 
		}
		child {
			node[node] {\footnotesize{$1$}}    
			child {
				node[leaf] {}     
				edge from parent [solid]
			}    
			edge from parent 
		}
		edge from parent [dashed]
	};
\end{tikzpicture}
\end{minipage}
\begin{minipage}{.05\textwidth}
\centering
$\longleftrightarrow$
\end{minipage}
\begin{minipage}{.45\textwidth}
\centering\tikzstyle{node} = [circle, draw, solid, text centered]
\tikzstyle{leaf} = [circle, minimum width=3pt,fill, inner sep=0pt]
\tikzstyle{level 1}=[level distance=1.5cm, sibling distance=1cm]
\tikzstyle{level 2}=[level distance=1.5cm, sibling distance=2.2cm]
\tikzstyle{level 3}=[level distance=1.5cm, sibling distance=1cm]
\tikzstyle{level 4}=[level distance=1.2cm, sibling distance=1cm]
\begin{tikzpicture}[grow'=up]
\node[leaf] {}
	child {
		node[node] {\footnotesize{$\susp x_1$}}        
		child {
			node[node] {\footnotesize{$\susp x_2$}}  
			child {
				node[node] {\footnotesize{$x_3$}}     
				child {
					node[leaf] {}     
					edge from parent [solid]
				}
				child {
					node[leaf] {}     
					edge from parent [solid]
				}
				child {
					node[leaf] {}     
					edge from parent [solid]
				}
				edge from parent 
			}
			child {
				node[leaf] {}
				edge from parent [dotted]
			}
			edge from parent   
		}
		child {
			node[node] {\footnotesize{$\susp x_4$}}     
			child {
				node[leaf] {}     
				edge from parent [solid]
			}
			child {
				node[leaf] {}     
				edge from parent [solid]
			}
			child {
				node[leaf] {}     
				edge from parent [solid]
			}
			edge from parent
		}
		child {
			node[leaf] {}    
			edge from parent [dotted]
		}
		edge from parent [dashed]
	};
\end{tikzpicture}
\end{minipage}
}
\caption{An element of $\Cyl(\sC)$ (on the left) identified with an element of $\OP(\sQ)$ (on the right). The dotted lines indicate edges of color $\gamma$.}\label{3-color-version}
\end{figure}

It is not hard to check that this identification is an isomorphism of cochain complexes, and consequently
\begin{align*}
\displaystyle H^\bullet(\Cyl(\sC)(n,0;\beta)) \cong \bigoplus_{m=0}^n H^\bullet (\OP(\sQ)(m,0,n-m;\beta))
\end{align*}

Since $\OP(\sQ)(m, 0, n-m; \beta)$ is $\colim$ from a finite, disjoint union of connected groupoids (specifically, the groupoids consisting of members of isomorphism classes of $3$-colored $n$-labeled planar trees), and carries only the differential structure coming from $\sQ$, Lemma \ref{colimma} applies. In particular, since taking coinvariants is exact when working over a field of characteristic $0$, we have from lemma \ref{colimma} that
\begin{align*}
\displaystyle H^\bullet (\OP(\sQ)(m,0,n-m;\beta)) = \OP(H^\bullet(\sQ))(m,0,n-m;\beta).
\end{align*}
But if $m > 0$, any element of $\OP(\sQ)(m,0,n-m;\beta)$ must contain at least one vertex decorated by an element of $\sQ(a,0,0;\beta)$. Since $H^\bullet (\sQ(a,0,0;\beta))=0$, we have in this case that ${ \OP(H^\bullet(\sQ))(m,0,n-m;\beta)=0}$ also. On the other hand, if $m=0$, all vertices are decorated by elements of $\sQ(0,b,c;\beta)$, and in this case we have $H^\bullet (\sQ(0,b,c;\beta)) = \sQ(0,b,c;\beta)$. Consequently,
\begin{align*}
\displaystyle H^\bullet(\Cyl(\sC)(n,0;\beta)) & \cong H^\bullet (\OP(\sQ)(0,0,n;\beta)) \\
& = \OP(H^\bullet (\sQ))(0,0,n;\beta) \\
& = \OP(\sQ)(0,0,n;\beta).
\end{align*}
Passing back to $\Cyl(\sC)(n,0;\beta)$ via the earlier isomorphism, we see that
\begin{align*}
\OP(\sQ)(0,0,n;\beta) \cong \iota_\beta({\Cobar}(\sC)(n)) \subseteq \Cyl(\sC)(n,0;\beta)
\end{align*}
which shows that $\iota_\beta$ is a quasi-isomorphism; therefore $\iota_\beta$ is a quasi-isomorphism for the original complexes, as desired.

It remains to show that $\iota_\alpha$ is homotopic to $\iota_\beta$ in $\Cyl(\sC)$ with the original differential $\del_0$. Observe that in $\Cyl(\sC)$, the presence of vertices of type $1^{\alpha\beta}$ is determined completely by the coloring of adjacent vertices, and whether they are decorated by suspended vectors or not. Therefore, given $X \in {\Cobar}(\sC)$, we may define $X_i \in \Cyl(\sC)$ by declaring that $X_i$ has the same underlying tree as $X$, it has the same internal vectors as $X$ but that the $i$th (nontrivial) vertex is no longer suspended (using the total order on vertices), that the edges before (nontrivial) vertex $i$ are of color $\beta$ and the edges after are of color $\alpha$ (using the total order on edges), and then finally adding trivial vertices $1^{\alpha\beta}$ and edges of color $\beta$ as necessary to make $X_i$ a valid element of $\Cyl(\sC)$. Figure \ref{coboundary} provides an example of this construction.

\begin{figure}[h]
\noindent\makebox[\textwidth][c]{
\begin{minipage}{.05\textwidth}
\centering
$X=$
\end{minipage}
\begin{minipage}{.4\textwidth}
\centering\tikzstyle{node} = [circle, draw, solid, text centered]
\tikzstyle{leaf} = [circle, minimum width=3pt,fill, inner sep=0pt]
\tikzstyle{level 1}=[level distance=1.5cm, sibling distance=1cm]
\tikzstyle{level 2}=[level distance=1.5cm, sibling distance=1.5cm]
\tikzstyle{level 3}=[level distance=1.5cm, sibling distance=1cm]
\tikzstyle{level 4}=[level distance=1.2cm, sibling distance=1cm]
\begin{tikzpicture}[grow'=up]
\node[leaf] {}
	child {
		node[node] {\footnotesize{$\susp x_1$}}        
		child {
			node[node] {\footnotesize{$\susp x_2$}}  
			child {
				node[node] {\footnotesize{$\susp x_3$}}     
				child {
					node[leaf] {}     
					edge from parent
				}
				child {
					node[leaf] {}     
					edge from parent
				}
				child {
					node[leaf] {}     
					edge from parent
				}
				edge from parent 
			}
			child {
				node[leaf] {}
				edge from parent
			}
			child {
				node[leaf] {}
				edge from parent
			}
			edge from parent   
		}
		child {
			node[leaf] {}    
			edge from parent
		}
		child {
			node[node] {\footnotesize{$\susp x_4$}}     
			child {
				node[leaf] {}     
				edge from parent
			}
			child {
				node[leaf] {}     
				edge from parent
			}
			child {
				node[leaf] {}     
				edge from parent
			}
			edge from parent
		}        
		edge from parent
	};
\end{tikzpicture}
\end{minipage}
\begin{minipage}{.05\textwidth}
\centering
$\longmapsto$
\end{minipage}
\begin{minipage}{.05\textwidth}
\centering
$X_2 =$
\end{minipage}
\begin{minipage}{.4\textwidth}
\centering\tikzstyle{node} = [circle, draw, solid, text centered]
\tikzstyle{leaf} = [circle, minimum width=3pt,fill, inner sep=0pt]
\tikzstyle{level 1}=[level distance=1.5cm, sibling distance=1cm]
\tikzstyle{level 2}=[level distance=1.5cm, sibling distance=1.7cm]
\tikzstyle{level 3}=[level distance=1.5cm, sibling distance=1cm]
\tikzstyle{level 4}=[level distance=1.2cm, sibling distance=1cm]
\begin{tikzpicture}[grow'=up]
\node[leaf] {}
	child {
		node[node] {\footnotesize{$\susp x_1$}}        
		child {
			node[node] {\footnotesize{$x_2$}}  
			child {
				node[node] {\footnotesize{$\susp x_3$}}     
				child {
					node[leaf] {}     
					edge from parent
				}
				child {
					node[leaf] {}     
					edge from parent
				}
				child {
					node[leaf] {}     
					edge from parent
				}
				edge from parent [solid]
			}
			child {
				node[leaf] {}
				edge from parent [solid]
			}
			child {
				node[leaf] {}
				edge from parent [solid]
			}
			edge from parent   
		}
		child {
			node[node] {\footnotesize{$1$}}
			child {
				node[leaf] {}    
				edge from parent [solid]
			}
			edge from parent
		}
		child {
			node[node] {\footnotesize{$1$}}
			child {
				node[node] {\footnotesize{$\susp x_4$}}     
				child {
					node[leaf] {}     
					edge from parent
				}
				child {
					node[leaf] {}     
					edge from parent
				}
				child {
					node[leaf] {}     
					edge from parent
				}
				edge from parent [solid]
			}
			edge from parent
		}        
		edge from parent [dashed]
	};
\end{tikzpicture}
\end{minipage}
}
\caption{The construction of $X_2 \in \Cyl(\sC)$ given $X \in {\Cobar}(\sC)$.}\label{coboundary}
\end{figure}

We may now construct the homotopy between $\iota_\alpha$ and $\iota_\beta$. Given $X=(\bt; \susp x_1,...,\susp x_k) \in {\Cobar}(\sC)$, define $h: {\Cobar}(\sC) \to \Cyl(\sC)$ by

\begin{align*}
h(X) = \sum_{i=1}^k (-1)^{|\susp x_1|+...+|\susp x_{i-1}|} \hspace{1mm} X_i .
\end{align*}
We may think of the sign in the above term coming from the suspension decorating the $i$th nodal vertex $x_i$ ``jumping over'' the vertices $\susp x_1, ..., \susp x_{i-1}$ to leave the tree. Since $\del_\sC (\susp x) = - \susp \del_\sC (x)$ for $x \in \sC$, we have that $\del_\sC \circ h + h \circ \del_\sC = 0$. It is also easy to check that $\iota_\alpha - \iota_\beta = (\del_0' + \del_0'') \circ h$; $\del_0'$ applied to the the first term of $h(X)$ yields $\iota_\alpha(X)$, $\del_0''$ applied to the last term of $h(X)$ yields $-\iota_\beta (X)$ (the sign from $h$ will cancel with the sign coming from $\del_0''$ ``jumping over'' the nontrivial vertices before the final vertex $x_k$), and all middle terms cancel from similar sign considerations. We therefore have in general that $\iota_\alpha - \iota_\beta = \del_0 \circ h + h \circ \del_\sC$, which shows that $\iota_\alpha$ and $\iota_\beta$ are homotopic, which completes the proof.

%

\end{proof}

In fact, Theorem \ref{incl-q-isom} is true with respect to the full differentials on ${\Cobar}(\sC)$ and $\Cyl(\sC)$, not simply the weight $0$ parts.

\begin{corollary}\label{full-incl-q-isom}

The inclusion maps
\begin{center}
\begin{tabular}{c c c c c}
$\iota_\alpha$, $\iota_\beta$ & $:$ & $({\Cobar} (\sC)(n), \del_{{\Cobar}})$ & $\tto$ & $ (\Cyl(\sC)(n,0;\beta), \del )$ \\
$\iota_\alpha$ & $:$ & $X$ & $\mapsto$ & $1^{\alpha\beta} \circ_1 X^\alpha$ \\
$\iota_\beta$ & $:$ & $X$ & $\mapsto$ & $\mu(X^\beta; 1^{\alpha\beta}, ..., 1^{\alpha\beta})$
\end{tabular}
\end{center}
are quasi-isomorphisms for all $n \geq 0$, and furthermore, are homotopic.

\end{corollary}

\begin{proof}

The argument that $\iota_\beta$ is a quasi-isomorphism is very similar, but requires an initial modification. First filter $\Cyl(\sC)(n,0;\beta)$ and ${\Cobar}(\sC)(n)$ by weight and form the associated graded complexes; this then gives us the exact situation of Theorem \ref{incl-q-isom}, and the result holds.

A different argument is needed to show that $\iota_\alpha$ and $\iota_\beta$ are homotopic. For this we introduce the map $\Pi: \Cyl(\sC)(n,0;\beta) \tto {\Cobar}(\sC)(n)$ defined as follows. If $X \in \Cyl(\sC)(n,0;\beta)$ contains any nontrivial mixed-color vertices (that is, vertices decorated by elements of $\sC_\circ$), $\Pi(X)=0$. Otherwise, define $\Pi(X)$ by changing all edges to color $\alpha$ and delete all trivial mixed vertices $1^{\alpha\beta}$, merging the adjacent edges; the result is an element of ${\Cobar}(\sC)$ because $X$ contained no nontrivial mixed vertices. Figure \ref{projection-ex} gives an example of this. It is easy to check that $\Pi$ is a map of cochain complexes and that $\Pi$ is a one-sided inverse to both $\iota_\alpha$ and $\iota_\beta$:
$$\Pi \circ \iota_\alpha = 1_{{\Cobar}(\sC)} = \Pi \circ \iota_\beta .$$

\begin{figure}[h]
\noindent\makebox[\textwidth][c]{
\begin{minipage}{.05\textwidth}
\centering
$\Pi :$
\end{minipage}
\begin{minipage}{.4\textwidth}
\centering\tikzstyle{node} = [circle, draw, solid, text centered]
\tikzstyle{leaf} = [circle, minimum width=3pt,fill, inner sep=0pt]
\tikzstyle{level 1}=[level distance=1.2cm, sibling distance=1cm]
\tikzstyle{level 2}=[level distance=1.5cm, sibling distance=3cm]
\tikzstyle{level 3}=[level distance=1.5cm, sibling distance=1cm]
\tikzstyle{level 4}=[level distance=1.2cm, sibling distance=1cm]
\begin{tikzpicture}[grow'=up]
\node[leaf] {}
	child {
		node[node] {\footnotesize{$\susp x_1$}}        
		child {
			node[node] {\footnotesize{$1$}}  
			child {
				node[node] {\footnotesize{$\susp x_2$}}     
				child {
					node[leaf] {}     
					edge from parent
				}
				child {
					node[leaf] {}
					edge from parent [solid]
				}
				child {
					node[leaf] {}
					edge from parent [solid]
				}
					edge from parent [solid]
			}
			edge from parent   
		}
		child {
			node[node] {\footnotesize{$\susp x_3$}}     
			child {
				node[node] {\footnotesize{$1$}}
				child {
					node[leaf] {}    
					edge from parent [solid]
				}  
				edge from parent
			}
			child {
				node[node] {\footnotesize{$1$}}
				child {
					node[leaf] {}    
					edge from parent [solid]
				}  
				edge from parent
			}
			child {
				node[node] {\footnotesize{$1$}}
				child {
					node[leaf] {}    
					edge from parent [solid]
				}  
				edge from parent
			}
			edge from parent
		}        
		edge from parent [dashed]
	};
\end{tikzpicture}
\end{minipage}
\begin{minipage}{.05\textwidth}
\centering
$\longmapsto$
\end{minipage}
\begin{minipage}{.4\textwidth}
\centering\tikzstyle{node} = [circle, draw, solid, text centered]
\tikzstyle{leaf} = [circle, minimum width=3pt,fill, inner sep=0pt]
\tikzstyle{level 1}=[level distance=1.5cm, sibling distance=1cm]
\tikzstyle{level 2}=[level distance=2.4cm, sibling distance=3cm]
\tikzstyle{level 3}=[level distance=1.5cm, sibling distance=1cm]
\begin{tikzpicture}[grow'=up]
\node[leaf] {}
	child {
		node[node] {\footnotesize{$\susp x_1$}}        
		child {
			node[node] {\footnotesize{$\susp x_2$}}     
			child {
				node[leaf] {}     
				edge from parent
			}
			child {
				node[leaf] {}
				edge from parent [solid]
			}
			child {
				node[leaf] {}
				edge from parent [solid]
			}
			edge from parent   
		}
		child {
			node[node] {\footnotesize{$\susp x_3$}}     
			child {
				node[leaf] {}     
				edge from parent
			}
			child {
				node[leaf] {}
				edge from parent [solid]
			}
			child {
				node[leaf] {}
				edge from parent [solid]
			}
			edge from parent   
		}
		edge from parent
	};
\end{tikzpicture}
\end{minipage}
}
\caption{A nontrivial example of the map $\Pi: \Cyl(\sC)(n,0;\beta) \tto {\Cobar}(\sC)(n)$.}\label{projection-ex}
\end{figure}

\noindent Since we already know that $\iota_\beta$ is a quasi-isomorphism, it follows that $\iota_\alpha$ and $\iota_\beta$ induce the same map on cohomology, and therefore are homotopic. 

\end{proof}

\section{Derivations and derived automorphisms of $\Cyl(\sC)$}

We now focus on derivations of operads \cite[Section 6.1]{NotesAlgOps}. The space of derivations of an operad is denoted $\Der(\sO)$, and is a cochain complex with differential given by the commutator bracket with the internal differential $\del_\sO$. With this in mind, we turn our attention to the dg Lie algebras $\Der({\Cobar}(\sC))$ and $\Der(\Cyl(\sC))$, and investigate how they relate to each other. Our ultimate goal is to show that given any derivation of ${\Cobar}(\sC)$, we can extend it to a derivation of $\Cyl(\sC)$; the ramifications of this in terms of $\infty$-morphisms will be discussed in the next section.

We first state and prove a technical lemma, to be used several times in the remainder of the paper.

\begin{lemma}\label{tech-lemma}
Let $\Hom(\widetilde{\sC}, \Cyl(\sC))$ be the cochain complex of maps of colored collections, with differential
\begin{align*}
\del(F) = \del_0 \circ F - (-1)^{|F|} F \circ \del_0
\end{align*}
for $F \in \Hom(\widetilde{\sC}, \Cyl(\sC))$. Define the cochain complex $\Hom(\susp \sC_\circ, \Cobar(\sC))$ similarly, with differential
\begin{align*}
\del(F) = \del_{\sC} \circ F - (-1)^{|F|} F \circ \del_{\sC}
\end{align*}
Then the maps
\begin{center}
\begin{tabular}{c c c c}
$\res_\alpha, \res_\beta :$ & $\Hom(\widetilde{\sC}, \Cyl(\sC))$ & $\tto$ & $\Hom(\susp \sC_\circ, \Cobar(\sC))$
\end{tabular}
\end{center}
given by restricting to a single color $\alpha$ or $\beta$ are quasi-isomorphisms of complexes.
\end{lemma}
\begin{remark}
By $F \circ \del_0$, we mean that $F \in \Hom(\widetilde{\sC}, \Cyl(\sC))$ acts on the nontrivial vertices that are present after applying $\del_0$. Explicitly, for a mixed-color generator $x^{\alpha\beta} \in \sC$,
\begin{align*}
(F \circ \del_0 )(x^{\alpha\beta}) = F(\del_\sC (x) ^{\alpha\beta}) + 1^{\alpha\beta} \circ_1 F (\susp x^\alpha) - \mu (F(\susp x^\beta); 1^{\alpha\beta}, ..., 1^{\alpha\beta} )
\end{align*}
\end{remark}
\begin{proof}
We restrict our attention to level $n$, and decompose $\Hom(\widetilde{\sC}, \Cyl(\sC))$ into subspaces (not subcomplexes) based on how a derivation acts on different color generators:
\begin{align*}
\Hom(\susp\sC_\circ(n), {\Cobar}(\sC)(n))^\alpha \oplus \Hom(\sC(n), \Cyl(\sC)(n,0;\beta))^{\alpha\beta} \oplus \Hom(\susp\sC_\circ(n), {\Cobar}(\sC)(n))^\beta.
\end{align*}
Here, the first summand gives the action of a derivation on corollas purely of color $\alpha$, the second summand on mixed-color corollas, and the third summand on corollas of color $\beta$; the superscripts make this explicit. Before we can state how the differential structure respects this decomposition, we need to recall the earlier maps
\begin{center}
\begin{tabular}{c c c c c}
$\iota_\alpha$, $\iota_\beta$ & $:$ & ${\Cobar} (\sC)(n)$ & $\tto$ & $ \Cyl(\sC)(n,0;\beta)$ \\
$\iota_\alpha$ & $:$ & $X$ & $\mapsto$ & $1^{\alpha\beta} \circ_1 X^\alpha$ \\
$\iota_\beta$ & $:$ & $X$ & $\mapsto$ & $\mu(X^\beta; 1^{\alpha\beta}, ..., 1^{\alpha\beta})$
\end{tabular}
\end{center}
and use them to define new, degree $1$ maps:
\begin{center}
\begin{tabular}{c}
$\incl_\alpha, \incl_\beta : \Hom(\susp\sC_\circ(n), {\Cobar}(\sC)(n)) \tto \Hom(\sC(n), \Cyl(\sC)(n,0;\beta))$ \vspace{2mm} \\
$\incl_\alpha (F)(x) = (-1)^{|F|}\iota_\alpha(F(\susp x))$ \vspace{2mm} \\
$\incl_\beta (F)(x) = (-1)^{|F|}\iota_\beta(F(\susp x))$
\end{tabular}
\end{center}
for $F \in \Hom(\susp\sC_\circ(n), {\Cobar}(\sC)(n))$ and $x \in \sC(n)$. It is then straightforward to check that with respect to the above decomposition of $\Hom(\widetilde{\sC}, \Cyl(\sC))$, $\del$ acts as follows:
\begin{align*}
\del(F+F'+F'') = \del(F) - \incl_\alpha(F) + \partial(F') + \incl_\beta(F'') + \del(F'')
\end{align*}
where
$$F \in \Hom(\susp\sC_\circ(n), {\Cobar}(\sC)(n))^\alpha$$
$$F' \in \Hom(\sC(n), \Cyl(\sC)(n,0;\beta))^{\alpha\beta}$$
$$F'' \in \Hom(\susp\sC_\circ(n), {\Cobar}(\sC)(n))^\beta$$ and where
\begin{center}
\begin{tabular}{c}
$\del(F)=\del_{\sC} \circ F - (-1)^{|F|} F \circ \del_{\sC}$ \vspace{2mm} \\
$\del(F')=\del_0 \circ F' - (-1)^{|F'|} F' \circ \del_0$ \vspace{2mm} \\
$\del(F'')=\del_{\sC} \circ F'' - (-1)^{|F''|} F'' \circ \del_{\sC}$.
\end{tabular}
\end{center}
As a cochain complex, $\Hom(\widetilde{\sC}, \Cyl(\sC))$ is therefore a ``cylinder-type construction'' as described in \cite[Appendix A]{DefComHI}, and the maps $\res_\alpha$ and $\res_\beta$ are the natural projections onto the first and third summands. By the same reference, it is enough to show that the maps
\begin{align*}
\susp^{-1}\incl_\alpha, \susp^{-1}\incl_\beta : \Hom(\susp\sC_\circ(n), {\Cobar}(\sC)(n)) \tto \susp^{-1}\Hom(\sC(n), \Cyl(\sC)(n,0;\beta))
\end{align*}
are quasi-isomorphisms. But this is precisely the situation obtained by applying the functor $\Hom (\susp\sC_{\circ}, - )$ to the maps
\begin{align*}
\iota_\alpha, \iota_\beta : {\Cobar} (\sC)(n) \tto \Cyl(\sC)(n,0;\beta)
\end{align*}
and we know those maps are quasi-isomorphisms from Theorem \ref{incl-q-isom}. Since $\Hom_{S_n}$ is exact when working over a field of characteristic $0$, $\incl_\alpha$ and $\incl_\beta$ are also quasi-isomorphisms. From the results of \cite[Appendix A]{DefComHI}, we conclude that the maps $\res_\alpha$, $\res_\beta$ are quasi-isomorphisms.
\end{proof}

\begin{proposition}\label{tech-lemma-filter}
Introduce the following descending filtration on $\Hom(\widetilde{\sC}, \Cyl(\sC))$:
\begin{align*}
\sF_m \Hom(\widetilde{\sC}, \Cyl(\sC)) = \{ F \in \Hom(\widetilde{\sC}, \Cyl(\sC)) \hspace{2mm} | \hspace{2mm} \wt(F) \geq m \}.
\end{align*}
Introduce the same filtration on $\Hom(\susp\sC_\circ, {\Cobar}(\sC))$. These filtrations are complete and compatible with the appropriate differentials and the restriction maps $\res_\alpha$, $\res_\beta$. Then
\begin{center}
\begin{tabular}{c c c c}
$\res_\alpha, \res_\beta :$ & $\sF_m \Hom(\widetilde{\sC}, \Cyl(\sC))$ & $\tto$ & $\sF_m \Hom(\susp \sC_\circ, \Cobar(\sC))$
\end{tabular}
\end{center}
remain quasi-isomorphisms.
\end{proposition}
\begin{proof}
The proof is exactly the same as that of Lemma \ref{tech-lemma}, restricting to the appropriate filtration levels.
\end{proof}

With the above results, we can now show our first main result, that restricting derivations of $\Cyl(\sC)$ to derivations of $\Cobar(\sC)$ yields quasi-isomorphisms of Lie algebras. This is the key statement needed for later results concerning derived automorphisms and their action on $\infty$-morphisms.

\begin{theorem}\label{der-q-isom}
The maps
\begin{center}
\begin{tabular}{c c c c}
$\res_\alpha, \res_\beta :$ & $\Der(\Cyl(\sC))$ & $\tto$ & $\Der({\Cobar}(\sC))$
\end{tabular}
\end{center}
given by restricting to a single color $\alpha$ or $\beta$ are homotopic quasi-isomorphisms of dg Lie algebras at all filtration levels.
\end{theorem}

\begin{proof}


We will show the result for the entire derivation algebras; the argument for a specific filtration level is almost identical, the only difference being restricting to the appropriate filtration level and using Proposition \ref{tech-lemma-filter}.

It is clear that the above restriction maps are morphisms of dg Lie algebras. Since derivations are uniquely determined by their action on generators, we may equivalently consider the maps
\begin{center}
\begin{tabular}{c c c c}
$\res_\alpha, \res_\beta :$ & $\Hom(\widetilde{\sC}, \Cyl(\sC))$ & $\tto$ & $\Hom(\susp\sC_\circ, {\Cobar}(\sC))$
\end{tabular}
\end{center}
still determined by restricting to a single color $\alpha$ or $\beta$. Here, the differentials on $\Hom(\widetilde{\sC}, \Cyl(\sC))$ and $\Hom(\susp\sC_\circ, {\Cobar}(\sC))$ take the following form:

$$ \del(F) = \del \circ F -(-1)^{|F|} \widehat{F} \circ \del $$

\noindent where the map $F$ defined on generators extends uniquely to the derivation $\widehat{F}$. This way, the identification from derivations to morphisms respects the differential structure.


Filter $\Hom(\widetilde{\sC}, \Cyl(\sC))$ and $\Hom(\susp\sC_\circ, {\Cobar}(\sC))$ by weight, as in the above Corollary; note that we could have equivalently defined these filtrations on $\Der(\Cyl(\sC))$ and $\Der({\Cobar}(\sC))$. These filtrations are complete and compatible with the appropriate differentials and the restriction maps $\res_\alpha$, $\res_\beta$. As before, we will move to the associated graded complexes which carry simpler differentials; from Lemma E.1 of \cite{OpTwistApp}, it suffices to show that the restriction maps are quasi-isomorphisms in this simpler setting. 
When we move to the associated graded complexes for this filtration, only the part of the differentials coming from the weight $0$ part of the internal differentials survives. Explicitly, $\Hom(\widetilde{\sC}, \Cyl(\sC))$ carries the reduced differential
\begin{align*}
\del(F) = \del_0 \circ F - (-1)^{|F|} \widehat{F} \circ \del_0
\end{align*}
for $F \in \Hom(\widetilde{\sC}, \Cyl(\sC))$, and $\Hom(\susp\sC_\circ, {\Cobar}(\sC))$ carries the differential
\begin{align*}
\del(F) = \del_{\sC} \circ F - (-1)^{|F|} \widehat{F} \circ \del_{\sC}
\end{align*}
for $F \in \Hom(\susp\sC_\circ, {\Cobar}(\sC))$. This is exactly the same situation as in Lemma \ref{tech-lemma}; and so $\res_\alpha, \res_\beta$ are quasi-isomorphisms on the associated graded level, and therefore in general as well.

To show that $\res_\alpha$ and $\res_\beta$ are homotopic, we will show that they induce the same map on cohomology. Recall the map $\Pi : \Cyl(\sC)(n,0;\beta) \to {\Cobar}(\sC)(n)$ from the proof of Corollary \ref{full-incl-q-isom}. Given closed $D \in \Der(\Cyl(\sC))$, define $T \in \Der({\Cobar}(\sC))$ on generators by
$$T(\susp x) = (-1)^{|D|} ( \Pi \circ D) (x^{\alpha\beta}).$$
$D$ is closed, so in particular
$$0 = [ \del, D](x^{\alpha\beta}) = (\del \circ D)(x^{\alpha\beta}) - (-1)^{|D|} (D \circ \del )(x^{\alpha\beta}).$$
If we rearrange the above terms and apply $\Pi$ we obtain the equation
$$(\Pi \circ D \circ \del)(x^{\alpha\beta}) = (-1)^{|D|} (\Pi \circ \del \circ D)(x^{\alpha\beta}) = (-1)^{|D|} (\del \circ \Pi \circ D)(x^{\alpha\beta}) = (\del \circ T)(\susp x)$$
recalling that $\Pi$ is a cochain map. It is straightforward to check that 
$$(\Pi \circ D \circ \del)(x^{\alpha\beta}) = (\res_\alpha D - \res_\beta D - (-1)^{|D|} (T \circ \del)) (sx)$$
and so we substitute this into the previous equation and rearrange terms to see that
$$(\res_\alpha D - \res_\beta D)(sx) = (\del \circ T + (-1)^{|D|} T \circ \del)(sx) = \del ( T )(sx).$$
Thus $\res_\alpha$ and $\res_\beta$ induce the same map on cohomology, and hence are homotopic.
\end{proof}

%

We now turn our attention to derivations that may be exponentiated to operad automorphisms. Define
\begin{align*}
\Der'(\Cobar(\sC)) = \sF_1 \Der({\Cobar}(\sC))
\end{align*}
to be the dg Lie algebra of derivations that raise the number of internal vertices by at least one. Equivalently, $\Der'(\Cobar(\sC))$ may be defined as consisting of derivations $D$ that satisfy
\begin{align*}
p_{\susp \sC_\circ} \circ D = 0
\end{align*}
where $p_{\susp \sC_\circ}$ is the canonical projection $\Cobar(\sC) \to \susp \sC_\circ$. This definition, in addition to our standing assumption of working with reduced cooperads $\sC$, ensures that the dg Lie algebra $\Der'(\sC)$ is pronilpotent.

\begin{proposition}\label{exp-cobar}
Given a degree 0 derivation $D \in \Der'({\Cobar} (\sC))$ for a cooperad $\sC$, $D$ is \emph{locally nilpotent}: for all $X \in {\Cobar}(\sC)$, $D^m(X) = 0$ for some $m \geq 0$. Consequently, assuming $D$ is a cocycle, we may exponentiate $D$ to an automorphism of ${\Cobar}(\sC)$,
\begin{align*}
\displaystyle \exp(D) = \sum_{m=0}^{\infty} \frac{1}{m!} D^m.
\end{align*}
Here, we use the same filtration as in the proof of Theorem \ref{der-q-isom}.
\end{proposition}
\begin{proof}
It is a standard result that a locally nilpotent derivation exponentiates to an automorphism, so it is enough to show that, under our assumptions, all derivations $D \in \Der'({\Cobar} (\sC))$ are locally nilpotent. This follows from straightforward weight considerations, since $D$ raises weight by at least $1$, and since all nodal vertices of ${\Cobar}(\sC)$ have at least $2$ incoming edges for reduced $\sC$.
\end{proof}

We have an identical result for derivations of $\Cyl(\sC)$.

\begin{proposition}\label{exp-cyl}
Given a degree 0 cocycle $\widetilde{D} \in \Der'(\Cyl (\sC))$ for a cooperad $\sC$, $\widetilde{D}$ is locally nilpotent, and therefore may be exponentiated to an automorphism of $\Cyl(\sC)$.
\end{proposition}
\begin{proof}
The same weight argument works here as for ${\Cobar}(\sC)$, with the observation that for a vertex to have only a single incoming edge, it must be $1^{\alpha\beta} \in \Cyl(\sC)(1,0;\beta)$.
\end{proof}
See Appendix \ref{app-der-aut} for a discussion of these notions for the more general case of quasi-free operads. Following this Appendix, define
$$\Aut'(\Cobar(\sC)) = \{ \varphi \in \Aut(\Cobar(\sC)) \hspace{2mm}|\hspace{2mm} \pi \circ \varphi |_{\susp\sC} = \id_{\susp\sC} \}$$
and
$$\Aut'(\Cyl(\sC)) = \{ \varphi \in \Aut(\Cyl(\sC)) \hspace{2mm}|\hspace{2mm} \pi \circ \varphi |_{\widetilde{\sC}} = \id_{\widetilde{\sC}} \},$$
which are the groups obtained by exponentiating $Z^0 \Der'(\Cobar (\sC))$ and $Z^0 \Der'(\Cyl (\sC))$, respectively. The earlier restriction maps induce obvious group homomorphisms, and the results from Appendix \ref{app-der-aut} show that these group homomorphisms are isomorphisms on homotopy classes of automorphisms.

\begin{theorem}\label{res-isom-h}
The group homomorphisms
\begin{center}
\begin{tabular}{c c c c}
$\res_\alpha, \res_\beta :$ & $\Aut'(\Cyl(\sC))$ & $\tto$ & $\Aut'({\Cobar}(\sC))$
\end{tabular}
\end{center}
induce identical isomorphisms on homotopy classes:
\begin{center}
\begin{tabular}{c c c c}
$\res :$ & $\hAut'(\Cyl(\sC))$ & $\tto$ & $\hAut'({\Cobar}(\sC))$.
\end{tabular}
\end{center}
\end{theorem}
\begin{proof}
Since $\res_\alpha$ and $\res_\beta$ are homotopic dg Lie algebra quasi-isomorphisms from Theorem \ref{der-q-isom}, they induce a single group homomorphism
$$\res: H^0 (\Der'(\Cyl(\sC))) \tto H^0(\Der'(\Cobar(\sC)))$$
(where the group structure is given my the Campbell-Hausdorff formula, as explained in Appendix \ref{app-der-aut}). Using the isomorphisms of Proposition \ref{exp-log-h}, we immediately get the induced isomorphism
$$\res: \hAut(\Cyl(\sC)) \tto \hAut(\Cobar(\sC)).$$
\end{proof}

\section{Acting on $\infty$-morphisms}

We will now turn our attention to homotopy algebras, and show how the operadic techniques and results developed earlier may be applied to their study. In particular, we will use $\Cyl(\sC)$ to study $\infty$-morphisms.

\begin{proposition}\label{prop:cyl-alg}
A $\Cyl(\sC)$-algebra structure on a pair of dg vector spaces $(V,W)$ is equivalent to the following triple of data:
\begin{enumerate}
\item a map ${\Cobar} (\sC) \to \End_V$;
\item a map ${\Cobar} (\sC) \to \End_W$;
\item an $\infty$-morphism $V \leadsto W$.
\end{enumerate}
\end{proposition}

\begin{proof}
Following \cite{StableI}, given cochain complexes $V$ and $W$, let $\End_{V,W}$ be the $2$-colored endomorphism operad. The pair $V$, $W$, being algebras over $\Cyl(\sC)$ means that there is a map of colored operads
\begin{align*}
F: \Cyl(\sC) \tto \End_{V,W}.
\end{align*}
By construction, the single-color portions of the above map correspond exactly to maps
\begin{align*}
F_\alpha: {\Cobar}(\sC) \tto \End_V \\
F_\beta: {\Cobar}(\sC) \tto \End_W.
\end{align*}
Observe next that the mixed-color portion of $F$ can be expressed in terms of its components
\begin{align*}F_{\alpha\beta}(n): \Cyl(\sC)(n,0;\beta) \tto \End_{V,W}(n,0;\beta).
\end{align*}
Equivalently,
\begin{align*}
F_{\alpha\beta}(n): \sC(n) \tto \Hom_\k (V^{\ot n}, W).
\end{align*}
This is, in turn, equivalent to a map
\begin{align*}
U_n: (\sC(n) \ot V^{\ot n})^{S_n} \tto W.
\end{align*}
which extends uniquely to (and is uniquely determined by) a coalgebra map
\begin{align*}
U: \sC(V) \tto \sC(W).
\end{align*}
Finally, it is a straightforward check that the compatibility of $F$ with the differentials on $\Cyl(\sC)$ and $\End_{V,W}$ is equivalent to $U$ being a dg coalgebra map, respecting the coderivations $Q_V$ and $Q_W$ (which correspond to $F_\alpha$ and $F_\beta$).
\end{proof}

It is easy to see that, given an operad $\sO$, an $\sO$-algebra $V$ via the map $F: \sO \to \End_V$, and an endomorphism $\varphi$ of $\sO$, the composite $F \circ \varphi$ defines a new $\sO$-algebra structure on $V$ via pullback, which we will often denote $V^{\varphi}$. This also holds true for colored operads and algebras over them. Thus, given a pair $(V,W)$ that is an algebra over $\Cyl(\sC)$ via the operad morphism $\widetilde{F}:\Cyl(\sC)\to\End_{V,W}$, and given an endomorphism $\widetilde{\varphi}$ of $\Cyl(\sC)$, the morphism $\widetilde{F}^{\widetilde{\varphi}} = \widetilde{F} \circ \widetilde{\varphi}$ defines a new $\Cyl(\sC)$-algebra structure on $(V,W)$; the result is encoded in the diagram
\begin{align*}
U^{\widetilde{\varphi}}: V^{\widetilde{\varphi}_{\alpha}} \leadsto W^{\widetilde{\varphi}_{\beta}}
\end{align*}
where $\widetilde{\varphi}_{\alpha}$ and $\widetilde{\varphi}_{\beta}$ denote the obvious restriction maps of $\widetilde{\varphi}$ onto the first and second single-colored components, respectively (we will begin using this notation more generally, keeping in mind the earlier results about $\res_\alpha$ and $\res_\beta$ for derivations). Keeping with the above notation for consistency, we will henceforth decorate colored derivations, automorphisms, maps, etc. by tildes (e.g. $\widetilde{\varphi}$), and omit such decoration for derivations etc. on single-color objects.

This procedure is very general, but requires that we start with an endomorphism of $\Cyl(\sC)$, which may be difficult to construct. The goal of this section is to show how, in certain instances, to extend an automorphism of $\Cobar(\sC)$ to an automorphism of $\Cyl(\sC)$ in a well-controlled way. Therefore, given a diagram $U: V \leadsto W$, we will not only be able to modify $V$ and $W$ via the automorphism $\varphi$, but the $\infty$-morphism $U$ as well, in some coherent way.

Given an algebra structure map $F$ as above, we will often abuse notation slightly and write $F^D$ instead of $F^{\exp(D)}$ (likewise for the colored setting). We will similarly abbreviate the decoration of homotopy algebras, writing $V^D$ instead of $V^{\exp(D)}$, etc. At this point, we can state and prove the main theorem of the paper.

\begin{theorem}\label{main-thm}
Let $V$ and $W$ be ${\Cobar}(\sC)$-algebras for a cooperad $\sC$, and let $U: V \leadsto W$ be an $\infty$-morphism between them. Given a degree $0$ cocycle $D \in \Der'({\Cobar} (\sC))$, there exists a degree $0$ cocycle $\widetilde{D} \in \Der'(\Cyl(\sC))$ such that $D$, $\widetilde{D}_\alpha$, and $\widetilde{D}_\beta$ are cohomologous in $\Der'(\Cobar(\sC))$. Therefore we can construct
\begin{align*}
U^{\widetilde{D}}: V^{\widetilde{D}_\alpha} \leadsto W^{\widetilde{D}_\beta}
\end{align*}
such that $V^{\widetilde{D}_\alpha}$ is homotopy equivalent to $V^D$ and $W^{\widetilde{D}_\beta}$ is homotopy equivalent to $W^D$, and so that the linear term of $U$ is unchanged: $U^{\widetilde{D}}_{(0)} = U_{(0)}$.
\end{theorem}
\begin{proof}
The existence of $\widetilde{D}$ satisfying the above properties is provided by Theorem \ref{der-q-isom}, which says that the maps $\res_\alpha$, $\res_\beta$ are homotopic quasi-isomorphisms. Therefore the automorphism $\exp(\widetilde{D}) \in \Aut(\Cyl(\sC))$ may be used to modify $U: V \leadsto W$ to $U^{\widetilde{D}}: V^{\widetilde{D}_\alpha} \leadsto W^{\widetilde{D}_\beta}$, as explained earlier.

For the statement about homotopy equivalence, observe first that Proposition \ref{exp-log-h} says that exponentiating cohomologous derivations of $\Cobar(\sC)$ yields homotopic automorphisms (here, we use a path object notion of homotopy). Thus, $\exp(D)$ and $\exp(\widetilde{D})_\alpha$ are homotopic. Using \cite{3Hom} to link various notions of homotopy, we see that $\exp(D)$ and $\exp(\widetilde{D})_\alpha$ are cylinder homotopic in the sense of \cite{OpCobarCyl}. Therefore, Theorem 5.2.1 of \cite{OpCobarCyl} says that there is an $\infty$-quasi-isomorphism
\begin{align*}
\Phi_V: V \leadsto V^{\widetilde{D}_\alpha}.
\end{align*}
Similarly, using that $D$ and $\widetilde{D}_\beta$ are cohomologous, we deduce the existence of an $\infty$-quasi-isomorphism
\begin{align*}
\Phi_W: W \leadsto W^{\widetilde{D}_\beta}.
\end{align*}
Finally, the statement that the linear terms of $U^{\widetilde{D}}$ and $U$ coincide follows from the fact that we are modifying $U$ via exponentiated derivations, which necessarily start with the identity. Since all derivations considered raise weight by at least $1$, all linear terms remain unchanged.
\end{proof}
In \cite{TamConGRT}, we need the above theorem exactly, in particular that the $\infty$-morphism $U^{\widetilde{D}}$ comes from an exponentiated derivation. The following corollary may be useful in other applications, and may be seen as a more full answer to the motivating question.

\begin{corollary}
Let $V$ and $W$ be ${\Cobar}(\sC)$-algebras for a cooperad $\sC$, and let $U: V \leadsto W$ be an $\infty$-morphism between them. Given a degree $0$ cocycle $D \in \Der'({\Cobar} (\sC))$, there exists an $\infty$-morphism
\begin{align*}
U' : V^D \leadsto W^D
\end{align*}
\end{corollary}
\begin{proof}
Using the same notation as in the proof of Theorem \ref{main-thm}, let
\begin{align*}
U' = {\Phi_W}^{-1} \circ U^{\widetilde{D}} \circ \Phi_V : V^D \leadsto W^D,
\end{align*}
where ${\Phi_W}^{-1}: W^{\widetilde{D}_\beta} \leadsto W^D$ is a homotopy inverse to $\Phi_W$ in the sense of Section 10.4 of \cite{AlgOps}.
\end{proof}

We may also study how the construction behaves when iterated, and see that the result is straightforward. For simplicity of notation, we will just focus on the $\infty$-morphisms - the changes on the source/target algebras should be clear.

\begin{proposition}
Let $U: V \leadsto W$ be as above, and $D_1, D_2 \in \Der'(\Cobar(\sC))$ degree zero closed derivations that give $\widetilde{D}_1, \widetilde{D}_2 \in \Der'(\Cyl(\sC))$. Then, using the above notation, 
\begin{align*}
(U^{\widetilde{D}_1})^{\widetilde{D}_2} = U^{\CH(\widetilde{D}_1, \widetilde{D}_2)}
\end{align*}
where $\CH(x,y)$ denotes the Campbell-Hausdorff series in the symbols $x$ and $y$.
\end{proposition}
\begin{proof}
If $U: V \leadsto W$ is an algebra over $\Cyl(\sC)$ via $F: \Cyl(\sC) \to \End_{V,W}$, we have
\begin{align*}
F \circ \exp(\widetilde{D}_1) \circ \exp(\widetilde{D}_2) = F \circ \exp(\CH(\widetilde{D}_1, \widetilde{D}_2))
\end{align*}
which gives the desired formula.
\end{proof}

As an example of a situation in which $H^0 ( \Der' ({\Cobar} (\sC)))$ is known to be nonzero, \cite{KonGrGRT} gives that $H^0 ( \Der' ({\Cobar} ( \Ger^{\vee}))) \cong \mathfrak{grt}$, the Grothendieck-Teichm\"{u}ller Lie algebra (see Section 4.2 of \cite{NotesAlgOps} for information on $\Ger^{\vee}$). This leads to an application of Theorem \ref{main-thm} to justify a statement made in Section 10.2 of \cite{KonGrGRT}, concerning $\text{GRT}_1$-equivariance of Tamarkin's construction of formality morphisms. While the full explanation is outside the scope of this paper, the idea is that while \cite{KonGrGRT} details how $\text{GRT}_1$ acts on homotopy algebras, it does not explicitly explain how it acts on $\infty$-morphisms, a gap filled by our Theorem \ref{main-thm}. The paper \cite{TamConGRT} deals with this situation fully; alternately, in the stable setting, this question will be addressed in \cite{StableII}.

\subsection{Homotopy uniqueness}

Given that the only choices in the proof of Theorem \ref{main-thm} involve cohomologous derivations, the result should be ``unique up to homotopy'' in some sense. We give two characterizations of this uniqueness; a full answer would rely on a more appropriate theoretical framework, such as the theory of model categories or $\infty$-categories, that are outside the scope of this paper.

First, and most obviously, we can reinterpret Theorem \ref{res-isom-h} as saying the following:

\begin{proposition}
Any automorphism $\widetilde{\varphi} \in \Aut'(\Cyl(\sC))$ is uniquely determined up to homotopy by its restriction onto either color. In particular, if Theorem \ref{main-thm} produces $\exp(\widetilde{D})$ such that $\exp(D)$ is homotopic to either $\widetilde{\varphi}_\alpha$ or $\widetilde{\varphi}_\beta$, then $\exp(\widetilde{D})$ is homotopic to  $\widetilde{\varphi}$.
\end{proposition}

Homotopic automorphisms will yield homotopic operad maps $\Cyl(\sC) \to \End_{V,W}$, and so the second characterization of homotopy uniqueness involves unraveling exactly what results from homotopic structure maps, in terms of the resulting homotopy algebras and $\infty$-morphisms. This should be viewed as a $2$-colored extension of the result from \cite{OpCobarCyl} that homotopic structure maps $\Cobar(\sC) \to \End_V$ yield homotopy equivalent algebras.

\begin{proposition}
Let $F, G: \Cyl(\sC) \to \End_{V,W}$ be maps of operads, corresponding respectively to the homotopy algebras and $\infty$-morphisms
$$ U_F: V_F \leadsto W_F$$
$$ U_G: V_G \leadsto W_G.$$
Suppose $F$ is homotopic to $G$. Then we obtain $\infty$-quasi-isomorphisms
$$\Phi: V_F \leadsto V_G$$
$$\Psi: W_F \leadsto W_G$$
such that $\Psi \circ U_F$ is homotopic to $U_G \circ \Phi$. That is, the following diagram of $\infty$-morphisms commutes up to homotopy:
\begin{center}
\begin{tikzpicture}[descr/.style={fill=white,inner sep=1.5pt}]
        \matrix (m) [
            matrix of math nodes,
            row sep=4 em,
            column sep=4 em,
            text height=1.5ex, text depth=0.25ex
        ]
        {V_F & W_F   \\
	 V_G & W_G \\
        };
        \draw[->, font=\scriptsize,>=latex]
        (m-1-1) edge node[auto]{$U_F$} (m-1-2)
        (m-2-1) edge node[auto]{$U_G$} (m-2-2)
        (m-1-1) edge node[auto]{$\Phi$} (m-2-1)
        (m-1-2) edge node[auto]{$\Psi$} (m-2-2);
\end{tikzpicture}
\end{center}
\end{proposition}
\begin{proof}
Let
$$\sH: \Cyl(\sC) \to \End_{V,W} \ot \Omega^\bullet(\kk)$$
be the operadic homotopy between $F$ and $G$:
$$\sH|_{t=0, dt=0} = F, \hspace{3mm} \sH|_{t=1, dt=0} = G .$$
By restricting $\sH$ to color $\alpha$, we see that $F_\alpha$ is homotopic to $G_\alpha$. As explained more fully in the proof of Theorem \ref{main-thm}, this yields homotopy equivalent algebras $\Phi: V_F \leadsto W_F$ \cite{OpCobarCyl}. Restricting to color $\beta$, we similarly obtain $\Psi$.

Let us express
$$\sH(t,dt) = \sH_0(t)+\sH_1(t)dt$$
so that $\sH$ being a map of operads is equivalent to the following:
\begin{enumerate}
\item For all $t$, $\sH_0$ is a map of operads $\Cyl(\sC) \to \End_{V,W}$
\item For all $t$, $\sH_1$ is a derivation relative to $\sH_0$
\item $\frac{d}{dt}\sH_0 = \del_{\End_{V,W}} \circ \sH_1 + \sH_1 \circ \del_{\Cyl(\sC)}$
\end{enumerate}
(see also Proposition \ref{exp-log-h}). We will use the above data to construct an explicit homotopy between $\Psi \circ U_F$ and $U_G \circ \Phi$, that is, a map of cofree $\sC$-coalgebras
$$H: \sC(V_F) \tto \sC(W_G \ot \Omega^\bullet(\kk))$$
such that
$$H|_{t=0, dt=0} = \Psi \circ U_F , \hspace{3mm} H|_{t=1, dt=0} = U_G \circ \Phi .$$
Here, we are using the notation $\sC(V_F)$ to denote the cofree $\sC$-coalgebra on $V$, with differential coming from $F_\alpha : \Cobar(\sC) \to \End_V$, and similarly elsewhere. If we also write
$$H(t,dt) = H_0(t)+H_1(t)dt,$$
we see the $H$ must satisfy appropriate parallel conditions as $\sH$ listed above; $H_0$ is a map of coalgebras, $H_1$ is a coderivation relative to $H_0$, and the correct similar condition on $\frac{d}{dt}H_0$.

For all $t$, $\sH$ yields the following: a diagram of homotopy algebras
$$U_t: V_t \leadsto W_t$$
along with $\infty$-quasi-isomorphisms
$$\Phi_t: V_F \leadsto V_t$$
$$\Psi_{1-t}: W_t \leadsto W_G.$$
Just as for coderivations, relative coderivations are uniquely determined by their composition with the canonical projection, so we may define a coderivation relative to $U_t$ by the following formula:
$$P_t : \sC(V_t) \leadsto W_t$$
$$P_t(X; v_1, ..., v_n) = \sH_1(X^{\alpha\beta})(v_1, ..., v_n)$$
where $(X; v_1, ..., v_n) \in \sC(V_t)$.

Since $P_t$ is a coderivation relative to the coalgebra map $U_t$, $\Psi_{1-t} \circ P_t \circ \Phi_t$ is a coderivation relative to $\Psi_{1-t} \circ U_t \circ \Phi_t$. It is then easy to check that
$$H_0(t) = \Psi_{1-t} \circ U_t \circ \Phi_t$$
$$H_1(t) = \Psi_{1-t} \circ P_t \circ \Phi_t$$
satisfy the required properties for $H=H_0 + H_1 dt$ to be the desired homotopy between $\Psi \circ U_F$ and $U_G \circ \Phi$.
\end{proof}

\appendix
\section{Colim from connected groupoids}

\begin{lemma}\label{colimma}
Let $F: \fg \to \cCh$ be a functor from a connected groupoid $\fg$. Then $\colim F = F(a)_{\Aut(a)}$, for any object $a \in \fg$.
\end{lemma}
\begin{proof}
Choose $a\in \fg$; we need to show that $F(a)_{\Aut(a)}$ is a co-cone for $F: \fg \to \cCh$, and that it is universal. That is, for any other co-cone $X$ for $F$, there is a unique map $\tau:F(a)_{\Aut(a)} \to X$, and for any $a,b\in \fg$ there are maps $\pi_b: F(b) \to F(a)_{\Aut(a)}$ and $\pi_c : F(c) \to F(a)_{\Aut(a)}$, such that the following diagram commutes:
\begin{center}
\begin{tikzpicture}[descr/.style={fill=white,inner sep=1.5pt}]
        \matrix (m) [
            matrix of math nodes,
            row sep=2.2 em,
            column sep=2.2 em,
            text height=1.5ex, text depth=0.25ex
        ]
        {  F(b) & & F(c)    \\
            & F(a)_{\Aut(a)} &     \\
            & X &   \\
        };
        \draw[->, font=\scriptsize,>=latex]
        (m-1-1) edge node[auto]{$F(g)$} (m-1-3)
        (m-1-1) edge node[auto,swap]{$\pi_b$} (m-2-2)
        (m-1-3) edge node[auto]{$\pi_c$} (m-2-2)
        (m-1-1) edge[out=-90,in=155] node[auto,swap]{$\psi_b$} (m-3-2)
        (m-1-3) edge[out=-90,in=25] node[auto]{$\psi_c$} (m-3-2)
        (m-2-2) edge node[auto]{$\tau$} (m-3-2);
\end{tikzpicture}
\end{center}
Note that we trivially have this for $b=c=a$, where $g$ is any automorphism of $a$. Then $\pi_b=\pi_c=\pi$, the canonical projection $F(a) \onto F(a)_{\Aut(a)}$, and $\tau$ exists and is unique because of the universal property of quotients. Since $\fg$ is a connected groupoid we have maps $h_{ba}: b \to a$ and $h_{ca}: c \to a$ (for simplicity, let the inverses of these maps be denoted $h_{ab}$ and $h_{ac}$, respectively). Then we have the commuting diagram
\begin{center}
\begin{tikzpicture}[descr/.style={fill=white,inner sep=1.5pt}]
        \matrix (m) [
            matrix of math nodes,
            row sep=2.2 em,
            column sep=2.2 em,
            text height=1.5ex, text depth=0.25ex
        ]
        {   F(b) & & F(c)    \\
            F(a) & & F(a)    \\
            & F(a)_{\Aut(a)} &     \\
            & X &   \\
        };
        \draw[->, font=\scriptsize,>=latex]
        (m-1-1) edge node[auto]{$F(g)$} (m-1-3)
        (m-1-1) edge node[auto]{$F(h_{ba})$} (m-2-1)
        (m-1-3) edge node[auto,swap]{$F(h_{ca})$} (m-2-3)
        (m-2-1) edge node[auto]{$F(h_{ca} g h_{ab})$} (m-2-3)
        (m-2-1) edge node[auto,swap]{$\pi$} (m-3-2)
        (m-2-3) edge node[auto]{$\pi$} (m-3-2)
        (m-2-1) edge[out=-90,in=155] node[auto]{$\psi_a$} (m-4-2)
        (m-2-3) edge[out=-90,in=25] node[auto,swap]{$\psi_a$} (m-4-2)
        (m-1-1) edge[out=-135,in=180] node[auto,swap]{$\psi_b$} (m-4-2)
        (m-1-3) edge[out=-45,in=0] node[auto]{$\psi_c$} (m-4-2)
        (m-3-2) edge node[auto]{$\tau$} (m-4-2);
\end{tikzpicture}
\end{center}
where the left and right triangles commute because $X$ is a co-cone for $F$, and the top rectangle commutes by construction. This then gives us the first diagram with $\pi_b=\pi \circ F(h_{ba})$ and $\pi_c=\pi \circ F(h_{ca})$, and therefore $\colim F = F(a)_{\Aut(a)}$.
\end{proof}

\section{On cohomologous derivations and homotopic automorphisms}\label{app-der-aut}
In this appendix we investigate the relationship between cohomologous derivations and homotopic automorphisms of operads. The techniques and results of this appendix were inspired by Appendix A of \cite{StablePolyvector}, where a similar situation was considered in the specific setting of $L_\infty$-algebras.

All operads will be (possibly colored) quasi-free dg operads $\sO = \OP(\sM)$ ($\sM$ a collection), equipped with the weight filtration as in Section 3 (our motivating exampes are $\Cobar(\sC)$ and $\Cyl(\sC)$ for a reduced cooperad $\sC$). Let $\pi$ be the projection $\sO \to \sM$. We then have the filtered dg Lie algebra $\Der(\sO)$ of operad derivations of $\sO$, which contains the subalgebra
$$\Der'(\sO) = \{ D \in \Der(\sO) \hspace{2mm}|\hspace{2mm} \pi \circ D |_{\sM}  = 0 \}.$$
We also have the group $\Aut(\sO)$ of operad automorphisms of $\sO$, with subgroup
$$\Aut'(\sO) = \{ \varphi \in \Aut(\sO) \hspace{2mm}|\hspace{2mm} \pi \circ \varphi |_{\sM} = \id_{\sM} \}.$$
With these conditions, we have well-defined maps
$$D \mapsto \exp(D) = \sum_{k \geq 0} \frac{D^n}{n!} : Z^0(\Der'(\sO)) \to \Aut'(\sO)$$
and
$$\varphi \mapsto \log(\varphi) = \sum_{k \geq 1} (-1)^{n-1} \frac{(\varphi - \id)^n}{n} : \Aut'(\sO) \to Z^0(\Der'(\sO))$$
that are inverse to each other (it is straightforward to show that degree $0$ closed derivations exponentiate to operad automorphisms, and vice versa).

We can make $Z^0( \Der'(\sO))$ into a group with composition given by the Campbell-Hausdorff formula
$$\CH(X,Y) = \log ( \exp(X) \exp(Y) )$$
and identity $0$ (it is an easy exercise that if $X, Y$ are degree $0$ and closed, so too is $\CH(X,Y)$). It will be clear from context whether we consider $Z^0 (\Der'(\sO))$ as a dg Lie algebra, or as a group. The following proposition is well known:

\begin{proposition}\label{exp-log}
The maps
$$\exp: Z^0(\Der'(\sO)) \to \Aut'(\sO)$$
$$\log: \Aut'(\sO) \to Z^0(\Der'(\sO))$$
are inverse group isomorphisms.
\end{proposition}

%

All of the above constructions and results are preserved when considering cohomologous derivations and homotopic automorphisms. It is straightforward to see that the group structure on $Z^0 \Der'(\sO)$ induces the group structure on $H^0 (\Der'(\sO))$. Recall \cite[Section 5.1]{NotesAlgOps} that two automorphisms $\varphi_1, \varphi_2 \in \Aut'(\sO)$ are homotopic if there exists an operad map
$$\sH: \sO \to \sO \ot \Omega^\bullet(\kk)$$
such that
$$\sH|_{t=0, dt=0} = \varphi_1, \hspace{3mm} \sH|_{t=1, dt=0} = \varphi_2$$
where $\Omega^\bullet(\kk)$ denotes the algebra of polynomial differential forms on $\kk$. We will also insist that such homotopies occur in $\Aut'(\sO)$, in the following sense. It is easy to see that for any specific choice of $t$, $\sH|_{dt=0}$ is an operad endomorphism of $\sO$ (see proof below); we additionally require that it be in $\Aut'(\sO)$. Let $\hAut'(\sO)$ denote the group of homotopy classes of automorphisms in $\Aut'(\sO)$. Then we have the following version of Proposition \ref{exp-log}:

\begin{proposition}\label{exp-log-h}
The induced maps
$$\exp: H^0(\Der'(\sO)) \to \hAut'(\sO)$$
$$\log: \hAut'(\sO) \to H^0(\Der'(\sO))$$
are inverse group isomorphisms.
\end{proposition}
\begin{proof}
This proof is essentially borrowed from Appendix A of \cite{StablePolyvector}. It must be shown that the above maps are well-defined; that they are inverse group isomorphisms is then essentially obvious.

To show that $\exp$ is well-defined, we will show that $\exp(\del(P))$ is homotopic to the identity in $\Aut'(\sO)$ for every degree $-1$ derivation $P \in \Der'(\sO)$. Let us denote by $t$ an auxiliary variable and consider the following map of dg operads:
$$\exp(t\del(P)): \sO \to \sO[t]$$
(this map lands in $\sO[t]$ for the same reasons that $\exp$ is well-defined in this situation). We have
$$\frac{d}{dt}\exp(t\del(P)) = \del(P) \circ \exp(t\del(P))$$
and hence the sum
$$\sH_P = \exp(t\del(P))+dt \hspace{1mm} P \circ \exp(t\del(P))$$
is a map of dg operads
$$\sH_P: \sO \to \sO \ot \Omega^\bullet (\kk).$$
It is clear that 
$$\sH_P|_{t=0, dt=0} = \id_\sO, \hspace{3mm} \sH_P|_{t=1, dt=0} = \exp(\del(P))$$
and therefore $\sH_P$ is the desired homotopy connecting $\id_\sO$ to the automorphism $\exp(\del(P))$.

To show that $\log$ is well-defined, we will show that if $\varphi \in \Aut'(\sO)$ is homotopic to $\id_\sO$ in $\Aut'(\sO)$, then $\log(\varphi)$ is exact. So assume that there is a homotopy
$$\sH: \sO \to \sO \ot \Omega^\bullet(\kk)$$
such that
$$\sH|_{t=0, dt=0} = \id_\sO, \hspace{3mm} \sH|_{t=1, dt=0} = \varphi.$$
$\sH$ necessarily has the form
$$\sH = \sH_0 + dt \hspace{1mm} \sH_1, \hspace{2mm}\text{where}\hspace{2mm} \sH_0, \sH_1: \sO \to \sO[t].$$
The compatibility of $\sH$ with the operadic multiplications is equivalent to the equations
$$ \sH_0(x \circ_i y) = \sH_0(x) \circ_i \sH_0(y) $$
$$ \sH_1(x \circ_i y) = \sH_1(x) \circ_i \sH_0(y) + (-1)^{|x|}  \sH_0(x) \circ_i \sH_1(y) $$
and compatibility with the differentials is equivalent to the equations
$$\del \circ \sH_0 = \sH_0 \circ \del$$
$$\frac{d}{dt} \sH_0 = \del \circ \sH_1 + \sH_1 \circ \del.$$
Extending $\sH_0, \sH_1$ linearly across $\kk[t]$ to get maps $\sH_0, \sH_1: \sO[t] \to \sO[t]$, we see that the previous equation implies that
$$\frac{d}{dt}\log(\sH_0) = \sum_{n=1}^\infty \frac{(-1)^{n-1}}{n} \sum_{m=0}^{n-1} (\sH_0 - \id)^m \circ [\del,\sH_1] \circ (\sH_0 - \id)^{n-m-1}$$
which can be rewritten as
$$\frac{d}{dt}\log(\sH_0) = [\del, Z]$$
where
$$Z=\sum_{n=1}^\infty \frac{(-1)^{n-1}}{n} \sum_{m=0}^{n-1} (\sH_0 - \id)^m \circ \sH_1 \circ (\sH_0 - \id)^{n-m-1}.$$
By integrating, this implies that $\log(\varphi) = \log(\sH_0)|_{t=1}$ is exact, provided that $Z$ is a derivation of $\sO[t]$. To show this, recall that for fixed $t$, $\sH_0$ is in $\Aut'(\sO)$. Consequently we can construct $\sH_0^{-1}$. Since $\sH_1$ is a derivation relative to $\sH_0$, $\sH_0^{-1}\sH_1$ is a derivation of $\sO[t]$. Now consider
$$\Psi(u)=\log(\sH_0\exp(u\sH_0^{-1}\sH_1)) = \sum_{n=1}^\infty \frac{(-1)^{n-1}}{n} (\sH_0\exp(u\sH_0^{-1}\sH_1) - \id)^n,$$
which is an element of $\Der'(\sO[t]) \widehat{\ot} \hspace{.5mm} \kk[u]$. Therefore  $\frac{d}{dt}\Psi(u)|_{u=0}$ is a derivation of $\sO[t]$, and it is straightforward to check that
$$ \frac{d}{dt}\Psi(u)|_{u=0} = \sum_{n=1}^\infty \frac{(-1)^{n-1}}{n} \sum_{m=0}^{n-1} (\sH_0 - \id)^m \circ \sH_1 \circ (\sH_0 - \id)^{n-m-1} = Z.$$
So $Z$ is indeed a derivation.

%

\end{proof}

\bibliography{references}

\end{document}